\documentclass{amsart}
\usepackage[margin = 1in]{geometry}

\usepackage{amsmath,amsthm,amssymb}
\usepackage{mathtools}
\usepackage{mathrsfs}
\usepackage{stmaryrd}
\usepackage{enumitem}
\usepackage{hyperref}
\usepackage{multicol}
\usepackage{mathdots}
\allowdisplaybreaks

\usepackage{tikz-cd}
\usepackage{caption}
\usepackage{subcaption}
\usetikzlibrary{decorations.markings,arrows,quotes,angles,positioning,calc}
\tikzset{vertex/.style={circle, minimum size=1pt, inner sep=1pt, fill}}

\newcommand{\C}{{\mathbb{C}}}
\newcommand{\R}{{\mathbb{R}}}
\newcommand{\Z}{{\mathbb{Z}}}
\newcommand{\N}{{\mathbb{N}}}

\renewcommand{\P}{{\mathcal{P}}}
\newcommand{\Pw}{{\P_w}}

\renewcommand{\t}{{\mathfrak{t}}}

\newcommand{\ui}{{\underline{i}}}
\newcommand{\uj}{{\underline{j}}}

\newcommand{\omu}{{\overline{\mu}}}
\newcommand{\olambda}{{\overline{\lambda}}}
\newcommand{\ogamma}{{\overline{\gamma}}}
\newcommand{\ow}{{\overline{w}}}

\newcommand{\trop}{{\mathrm{trop}}}

\newtheorem{theorem}{Theorem}[section]
\newtheorem{lemma}[theorem]{Lemma}
\newtheorem{proposition}[theorem]{Proposition}
\newtheorem{corollary}[theorem]{Corollary}

\theoremstyle{definition}
\newtheorem{definition}[theorem]{Definition}
\newtheorem{question}[theorem]{Question}
\newtheorem{example}[theorem]{Example}
\newtheorem{remark}[theorem]{Remark}

\newtheorem{claim}{Claim}
\newenvironment{claimproof}[1]{\par\noindent\underline{Proof:}\space#1}{\hfill $\blacksquare$}

\title{Lower affine MV polytopes of rank 2}
\author{Kathlyn Dykes}
\address{Tutte Institute of Mathematics and Computing}
\email{\texttt{kathlyndykes@gmail.com}}
\date{}

\begin{document}

\begin{abstract}
When $G$ is a complex reductive algebraic group, MV polytopes are in bijection with the non-negative tropical points of the unipotent group of $G$. In this paper, we prove a similar theorem for certain subclasses of rank 2 affine MV polytopes.  

For the Kac-Moody group $\widehat{SL_2}$, an affine MV polytopes splits into three subpolytopes: a lower, a middle and an upper polytope. The lower polytopes are natural generalizations of finite-type polytopes with highest vertex labelled by an arbitrary Weyl element. We extend the known results from this subclass of finite-type MV polytopes to the case of lower affine MV polytopes of rank 2. We prove that for an element $w$ of the affine Weyl group, the class of lower affine MV polytopes with highest vertex $w$ are in bijection with the non-negative tropical points of the reduced double Bruhat cell labelled by $w^{-1}$. To do this, we describe the BZ data of a rank 2 affine MV polytope and show that certain generalized minor functions satisfy the conditions of a lower polytope. As any upper polytope is a reflection of some lower polytope, a analogous result will hold for the class of upper affine MV polytopes of rank 2. 
\end{abstract}

\maketitle

\tableofcontents

\section{Introduction}

For $G$ a complex reductive algebraic group, the representations of the Langlands dual group $G^\vee$ are intricately related to the intersection homology of the affine Grassmanian by the geometric Satake correspondence from \cite{MirkovicVilonen2007}. By this correspondence, Mirkov\'{i}c and Vilonen define a canonical basis for each representation using certain subvarieties of the affine Grassmanian, called Mirkov\'{i}c-Vilonen (MV) cycles. To better understand these varieties, Anderson \cite{PolytopeCalc}  introduced the idea of studying their moment polytopes. In \cite{MVpolytopes}, Kamnitzer gives a complete description of these polytopes, called MV polytopes, providing a combinatorial framework to better understand MV cycles and the corresponding bases. 


When $G$ is an affine Kac-Moody group, the representation theory of $G$ is very similar to that of an algebraic group, so there is hope that the geometric theory of MV cycles will extend. 
Unfortunately, in this case the affine Grassmanian can no longer be viewed as a geometric object and the standard definition of an MV cycle is not well defined. In \cite{CoulombBranchesof}, Braverman, Finkelberg and Nakajima conjecture that geometric spaces called Coulomb branches will replace the affine Grassmanian in the generalization of the geometric Satake correspondence to the affine case. There is also an analogue of MV cycles, called double MV cycles, introduced by Muthiah \cite{DoubleMVcyles}. 

On the other hand, affine MV polytopes are well understood. They have been defined using preprojective algebras \cite{AffineMVpolytopes}, KLR algebras \cite{MVpolytopesandKLR} and affine PBW bases \cite{AffinePBWBases}. As in the finite case, affine MV polytopes have a purely combinatorial definition. The motivation for this paper is to contribute to the work of realizing affine MV polytopes as the tropical points of some variety, as \cite{Geometryofcanonical} had done for the finite case. More specifically, we would like to answer the following question. 
\begin{question}\label{question:1}
 When we take $G$ to be an affine Kac-Moody group, can we generalize the results of the finite case to find a bijection between the tropical points of a variety and affine MV polytopes?
\end{question}

We simplify this question to the case where $G = \widehat{SL_2}$ so that we only consider affine MV polytopes of rank 2. As Baumann, Kamnitzer and Tingley proved that affine MV polytopes are characterized by the fact that their 2-faces are rank 2 MV polytopes  \cite{AffineMVpolytopes}, an answer to this question for rank 2 affine MV polytopes will be the first step towards an understanding of the general case.

Even by considering only 2-dimensional affine MV polytopes, this question is complicated, as these polytopes have a decoration. Instead, we can split an affine MV polytope into three polytopes: a lower, a middle and an upper polytope. The middle polytope is four sided and comes with a decoration. Meanwhile the upper and lower polytopes are closely related to the subsets of MV polytopes studied in \cite{MVBruhat} and are a natural generalization of the finite case.

In this paper, we will answer a simplified version of Question \ref{question:1} by considering only upper and lower affine polytopes. First, we recall affine MV polytopes in Section \ref{section:background}. In Section \ref{section:BZdata}, we describe the BZ data of lower affine MV polytopes. In Section \ref{section:tropicalgeometry}, we describe the topical geometry of reduced double Bruhat cells and prove that the tropical points of these cells correspond to these lower affine MV polytopes. 

\section{Background}\label{section:background}

We consider the affine Kac Moody group of type $A_1^{(1)}$, denoted by $G = \widehat{SL_2}$, and its corresponding affine Lie algebra $\widehat{sl_2}$. Let $\alpha_0, \alpha_1$ be the simple roots where $\alpha_1$ is the root for $SL_2$ and $\alpha_0$ is the affine root. Let $\alpha^\vee_0, \alpha^\vee_1$ be the simple coroots, set $\delta = \alpha_0^\vee + \alpha_1^\vee$  and denote $\t_\R = \text{span}_{\R}\{\alpha_0^\vee, \alpha_1^\vee\}$. Let $\t_\R^*$ be the dual space and let $\langle \cdot, \cdot \rangle: \t_\R \times \t_\R^* \rightarrow \R$ be the natural pairing. Let $\omega_0, \omega_1$ be the fundamental weights, i.e. $\langle \alpha_i^\vee,\omega_j \rangle = \delta_{i,j}$.  Denote the set of roots by $\Delta$ and the coroots by $\Delta^\vee$. The set of positive coroots are $\Delta^\vee_+ = \{ \alpha^\vee_0 + i \delta, \alpha^\vee_1 + i \delta, (i+1) \delta : i \in \mathbb{Z}_{\geq0}\}$ while the set of negative coroots are $\Delta^\vee_- = - \Delta^\vee_+$. It follows that $\Delta^\vee = \Delta^\vee_+ \cup \Delta^\vee_-$.

Let $W$ be the affine Weyl group associated to $G$. The Weyl group can be viewed as the reflection group on the coroots generated by the simple reflections, $s_1, s_2$. These reflections act on the coroots by:
\begin{align*}
s_0(\alpha^\vee_0) = - \alpha^\vee_0, && s_0(\alpha^\vee_1) = \alpha^\vee_0 + \delta, && s_1(\alpha^\vee_0) = \alpha^\vee_1 + \delta, && s_1(\alpha^\vee_1) = - \alpha_1. 
\end{align*}

\begin{definition}
A convex polytope $P$ in span$_\Z\{\alpha^\vee_0, \alpha^\vee_1\}$ is an $\widehat{sl_2}$ \emph{GGMS polytope} if all edges are parallel to coroots in $\Delta^\vee$. 
\end{definition}

We can label the vertices of a GGMS polytope $P$ by $\mu_k, \mu^k, \omu_k, \omu^k$ for $ k \in \mathbb{N}$. For $k \geq 1$, there exists $a_k, a^k, \bar{a}_k, \bar{a}^k \in \N$ such that 
 \begin{align*}
 \mu_k - \mu_{k-1} &= a_k(\alpha^\vee_1 + (k-1)\delta),  & \mu^{k-1} - \mu^k &= a^k (\alpha^\vee_0 + (k-1) \delta)), \\
 \omu_k - \omu_{k-1} &= \bar{a}_k (\alpha^\vee_0 + (k-1) \delta), & \omu^{k-1} - \omu^k & = \bar{a}^k (\alpha^\vee_1 + (k-1)\delta).
 \end{align*}
The limits $\lim_{k \rightarrow \infty} \mu_k,\lim_{k \rightarrow \infty} \mu^k,\lim_{k \rightarrow \infty} \omu_k,\lim_{k \rightarrow \infty} \omu^k$ exist and we denote them by 
\begin{align}\label{equation:polytopelimits}
\lim_{k \rightarrow \infty} \mu_k = \mu_\infty, && \lim_{k \rightarrow \infty} \mu^k =  \mu^\infty,&& \lim_{k \rightarrow \infty} \omu_k =  \omu_\infty,  && \lim_{k \rightarrow \infty} \omu^k = \omu^\infty.
\end{align}  

We define a \emph{decorated GGMS polytope} as a GGMS polytope $P$ along with two sequences of non-negative integers $\lambda = (\lambda_1 \geq \lambda_2, \dots)$ and $\olambda = (\olambda_1 \geq \olambda_2, \dots)$ with  $\mu^\infty - \mu_\infty = |\lambda| \delta$ and $\omu^\infty - \omu_\infty = |\olambda| \delta$. 

We consider the MV polytopes of $P$ as defined in \cite{Rank2affine}:
\begin{definition}[{\cite[Definition 3.4]{Rank2affine}}]\label{definition:affinepolytope}
An $\widehat{sl_2}$ MV polytope $P$ is a decorated GGMS polytope such that:
\begin{enumerate}[label=(\roman*)]
 \item If the lines between $\mu_\infty, \omu_\infty$ and $\mu^\infty, \omu^\infty$ are parallel then $\lambda = \olambda$, otherwise one is obtained from the other by removing a part of size $\langle \mu_\infty - \omu_\infty, \omega_0 + \omega_1 \rangle $,
 \item $\lambda_1, \olambda_1 \leq \langle \mu_\infty - \omu_\infty, \omega_0 + \omega_1 \rangle$,
 \item\label{condition:lower} For $k \geq 2$, $\max \{ \langle \omu_k - \mu_{k-1},\omega_1 \rangle, \langle \mu_k - \omu_{k-1}, \omega_0 \rangle \}=0$,
 \item\label{condition:upper} For $k \geq 2$, $\min \{ \langle \omu^k - \mu^{k-1}, \omega_0 \rangle, \langle \mu^k - \omu^{k-1}, \omega_1 \rangle \}=0$.
\end{enumerate}
\end{definition}

For an example of one of these polytopes, see Figure \ref{polytope}. Note that $\langle \mu_\infty - \omu_\infty, \omega_0 + \omega_1 \rangle$ is the ``width'' of the polytope. 

Due to their importance, we will refer to the conditions \ref{condition:lower} and \ref{condition:upper} as the \emph{lower} and \emph{upper diagonal relations}, respectively. These conditions have a visual interpretation. 
Consider the line in the $\alpha_1^\vee$ direction passing through $\omu_{k-1}$ and the line in the $\alpha_0^\vee$ direction passing through $\mu_{k-1}$. The diagonal relation guarantees that $\mu_k$ is on or below the first line while $\omu_k$ is on or below the second line, with at least one of these vertices on their respective lines. If a vertex is on the diagonal, we call this line an \emph{active diagonal}. In Figure \ref{polytope}, the diagonal relations are shown as the blue lines, where the solid lines are active diagonals and the dotted lines are not active. 

\begin{figure}[ht]
\[
\begin{tikzpicture}
\node[vertex] (O) at (0,0) [label=below:$\mu_0$] {};
\node[vertex] (U1) at (1.2,0.3)[label=below right:$\mu_1$] {};
\node[vertex] (U2) at (3, 1.65)[label=below right:$\mu_2$] {};
\node[vertex] (U3) at (4.2, 3.15)[label=below right:$\mu_3$] {};
\node[vertex] (U4) at (4.8, 4.2) {};
\node[vertex] (V3) at (4.8, 5.4) {};
\node[vertex] (V2) at (3, 7.65)[label=above right:$\mu^2$] {};
\node[vertex] (V1) at (1.8, 8.55)[label=above right:$\mu^1$] {};
\node[vertex] (T) at (-0.6,9.15)[label=above:$\mu^0$] {};
\node[vertex] (Y1) at (-1.8,8.85)[label=above left:$\omu^1$] {};
\node[vertex] (Y2) at (-3.6, 7.5) {};
\node[vertex] (W3) at (-3.6, 2.1) {};
\node[vertex] (W2) at (-3, 1.35)[label=below left:$\omu_2$] {};
\node[vertex] (W1) at (-1.8, 0.45)[label=below left:$\omu_1$] {};

\node at (-5.45, 7.5) {$\omu^\infty = \cdots = \omu^3=\omu^2$};
\node at (-5.45, 2.1) {$\omu_\infty = \cdots = \omu_4=\omu_3$};
\node at (6.2, 5.4) {$\mu^3 = \cdots = \mu^\infty$};
\node at (6.2, 4.2) {$\mu_4 = \cdots = \mu_\infty$};

\draw (Y2) -- (V3);
\draw (W3) -- (U4);

\draw[thick] (O) -- (U1) -- (U2) -- (U3) -- (U4) -- (V3) -- (V2) -- (V1) -- (T) -- (Y1) -- (Y2) -- (W3) -- (W2) -- (W1) -- (O);

\draw[thick, blue] (U1) -- (W2);
\draw[thick, blue] (W1) -- (U2);

\draw[thick, blue] (W2) -- (U3);
\draw[thick, dotted, blue] (U2) -- (-3.6, 3.3);

\draw[thick, blue] (Y1) -- (V2);
\draw[thick, dotted, blue] (Y2) -- (1.6, 8.60);

\end{tikzpicture}
\]
\caption{A rank 2 affine MV polytope}
\label{polytope}
\end{figure}
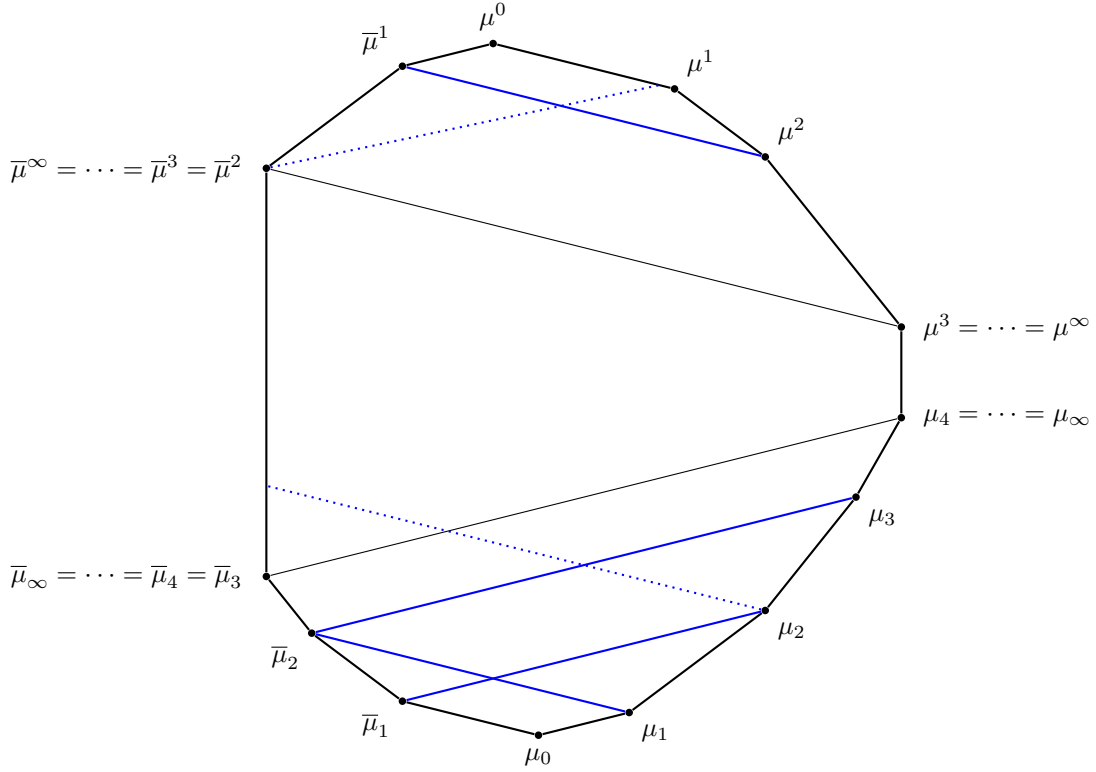

These polytopes naturally split into three sub-polytopes: a lower polytope, an upper polytope and a middle polytope. For an MV polytope $P$, we can define the lower and upper sub-polytopes of $P$ as
\begin{align*}
L(P) &= \text{conv} \{ \mu_k, \omu_k : k \in \N\}, & U(P) & = \text{conv}\{ \mu^k, \omu^k : k \in \N\}.
\end{align*}
The middle polytope is the convex hull $M(P) = \text{conv}\{ \mu_\infty, \mu^\infty, \omu_\infty, \omu^\infty \}$  along with the decorations $\lambda, \olambda$. 

\begin{lemma}\label{lemma:lumGGMS}
For an MV polytope $P$, the subpolytopes $L(P), U(P)$ and $M(P)$ are GGMS polytopes. 
\end{lemma}

\begin{proof}
The only thing we need to show is that any new edges in these subpolytopes are still in the direction of a coroot and have integer length. In $L(P)$, the only new edge we introduce is $\mu_\infty -\omu_\infty$. Let $k\in N$ be such that $\mu_k = \mu_\infty$ and $\omu_\infty = \omu_k$. First, notice that 
\begin{align*}
\mu_\infty - \omu_\infty &= \mu_k - \omu_k = \sum_{i=1}^k (\mu_{i} - \mu_{i-1}) - \sum_{i=1}^k (\omu_k - \omu_{k-1})\\&=  \sum_{i=1}^k a_i(\alpha_1^\vee + (i-1)\delta) - \sum_{i=1}^k \bar{a}_i(\alpha_0^\vee + (i-1) \delta)
\end{align*}
and thus $\langle \mu_\infty - \omu_\infty, \omega_0 + \omega_1 \rangle = \sum_{i=1}^k (2i-1)a_i - \sum_{i=1}^k  (2i-1) \bar{a}_k \in \Z$. 

The $k+1$ lower diagonal relation \ref{condition:lower} in Definition \ref{definition:affinepolytope} gives $\min \{ \langle \omu_\infty - \mu_\infty, \omega_1 \rangle, \langle \mu_\infty - \omu_\infty, \omega_0 \rangle \} =0$. Thus one of these is zero and hence the edge $\mu_{\infty} - \omu_\infty \in \Z \alpha_0^\vee$ or $\Z \alpha_1^\vee$.

A similar proof works to show that the edge $\mu^\infty - \omu^\infty$ in $U(P)$ is also in $\Z\alpha_0^\vee$ or $\Z\alpha_1^\vee$. As $M(P)$ shares the edges $\mu_\infty - \omu_\infty$ and $\mu^\infty - \omu^\infty$ with $U(P)$ and $L(P)$, then all edges of $M(P)$ are in simple coroot directions with lengths in $\Z$ as well. Thus these three subpolytopes are GGMS polytopes. 
\end{proof}

Note that the previous lemma does not necessarily hold for a GGMS polytope $P$ as these polytopes may not satisfy the diagonal relations. 

We can define subsets of affine MV polytopes inspired by these subpolytopes.  
\begin{definition}\label{definition:lower,upper,middle}
A \emph{lower affine MV polytope} is an affine MV polytope such that $\mu_\infty = \mu^0$ or $\omu_\infty = \mu^0$. An \emph{upper affine MV polytope} is an affine MV polytope such that $\mu^\infty = \mu_0$ or $\omu^\infty = \mu_0$. 

A \emph{middle affine MV polytope} is an affine MV polytope such that $\mu_\infty=\mu_0 $ or $\mu^\infty=\mu^0$ and $\omu_\infty=\mu_0$ or $\omu^\infty= \mu^0 $ along with the decoration $\lambda, \olambda$.
\end{definition}

We call $a_n$ the length of the edge between $\mu_n$ and $\mu_{n+1}$, while $\bar{a}_n$ is the length of the edge between $\omu_n$ and $\omu_{n+1}$. Similarly, let $a^n$ be the length of the edge between $\mu^n$ and $\mu^{n+1}$ and $\bar{a}^n$ the length of the edge between $\omu^n$ and $\omu^{n+1}$. As in the finite case, we can define the Lusztig data as the sequence of lengths along a minimal path from $\mu_0$ to $\mu^0$ in the polytope. 

\begin{definition}
The \emph{right Lusztig data} are $(a_n, \lambda_n, a^n)_{n \in \N}$. The \emph{left Lusztig data} are $(\bar{a}_n, \olambda_n, \bar{a}^n)_{n \in \N}$. 
\end{definition}

Although the vertex data $(\mu_k, \omu_k, \mu^k, \omu^k)_{k \in \N}$ of an affine MV polytope $P$ are an infinite collection, the polytope $P$ can be written as a convex hull of a finite set and hence only a finite subset of the vertex data will be distinct. As we will see in Section \ref{section:highestvertexpoly}, all lower affine MV polytopes have highest vertex $w$ for some Weyl group element $w$. Similarly, every upper affine MV polytope is an upper MV polytope of lowest vertex $v$ for some $v\in W$. 

\section{BZ data of affine MV polytopes}\label{section:BZdata}

We define the BZ data of affine polytopes in a similar way to the finite case in \cite{MVpolytopes}. First, suppose that $s_{i_1} s_{i_2} \dots s_{i_k}$ is a reduced word in the Weyl group. For an integer $k \geq 1$, there are exactly two Weyl group elements of length $k$: one starting with $s_0$ and the other starting with $s_1$. Denote the unique Weyl group element of length $k$ starting with $s_0$ as $\ow_k$ and denote the length $k$ element starting with $s_1$ as $w_k$. 

We will call weights of the form $s_{i_1} s_{i_2} \dots s_{i_k} \omega_{i_k}$ \textit{chamber weights of level $i_k$} and label the set of chamber weights by $\Gamma$. We will denote the weight $s_{i_1} s_{i_2} \dots s_{i_k}\omega_{i_k}$ by $\ogamma_{k+1}$ if $i_1 =0$, $\gamma_{k+1}$ when $i_1 =1$ and set $\gamma_1 = \omega_0$, $\ogamma_1=\omega_1$.

For a fixed $w \in W$, denote $\Gamma^w = \{ s_{i_1} \cdots s_{i_k} \omega_{i_k} : s_{i_1} \cdots s_{i_k} \leq w\} \cup \{ \omega_0, \omega_1\} \subset \Gamma$, where $\leq$ is the strong Bruhat order.

\begin{remark}\label{remark:gammaw}
As the root system is rank 2, these subsets of chamber weights are easy to describe. For $w = w_m$, these sets are given by: 
\begin{align*}
\Gamma^w = \{ \gamma_{k_1}, \ogamma_{k_2}: 1 \leq k_1 \leq m+1, 1 \leq k_2 \leq m \}, && \Gamma \setminus \Gamma^w = \{ \gamma_{k_1}, \ogamma_{k_2}: k_1 \geq m+2, k_2 \geq m+1\}.
\end{align*}
Similarly for $w = \ow_m$, these sets are given by:
\begin{align*}
\Gamma^w = \{ \gamma_{k_1}, \ogamma_{k_2}: 1 \leq k_1 \leq m, 1 \leq k_2 \leq m+1 \}, && \Gamma \setminus \Gamma^w = \{ \gamma_{k_1}, \ogamma_{k_2}: k_1 \geq m+1, k_2 \geq m+2\}.
\end{align*}
\end{remark} 

\begin{definition}\label{definition:affineBZdata}
Let $P$ be a GGMS polytope. Define the hyperplane data of $P$ as the collection $(M_\gamma, M^\gamma)_{\gamma \in \Gamma}$ where we define the lower hyperplane data $(M_\gamma)_{\gamma \in \Gamma}$ by $M_{\ogamma_k} = \langle \omu_k, \ogamma_k \rangle$ and $M_{\gamma_k} = \langle \mu_k, \gamma_k \rangle$. Define the upper hyperplane data $(M^\gamma)_{\gamma \in \Gamma}$ by $M^{\ogamma_k} = \langle \omu^k, - \gamma_k\rangle$ and $M^{\gamma_k} = \langle \mu^k, -\ogamma_k \rangle$.
\end{definition}

For ease of notation, we will denote $M_{\ogamma_k} = M_{\overline{k}}, M_{\gamma_k} = M_{k}, M^{\ogamma_k} =M^{\overline{k}}, M^{\gamma_k} = M^k$. Also, set $\gamma_k^* := \ogamma_k$ and $\ogamma_k^* := \gamma_k$. The hyperplane data of a polytope will define the polytope $P$. 

\begin{lemma}
For $P$ a GGMS polytope, $P = \{ x \in \t_\R: \langle x,  \gamma \rangle \leq M_\gamma, \langle x, -\gamma^* \rangle \leq M^\gamma,  \forall \gamma \in \Gamma\}$.
\end{lemma}

\begin{proof}
Let $P$ be a GGMS polytope. Let $\psi_P: \t_\R^* \rightarrow \R$ by $\psi_P(\theta) = \max_{x\in P}\langle x, \theta \rangle$. By \cite[Appendix]{MVpolytopes}, 
\[
P = \{ x \in \t_\R : \langle x, \theta \rangle  \leq \psi_P(\theta), \, \forall \theta \in \t_\R^* \}.
\]
We can apply \cite[Proposition 2.5, 2.16]{AffineMVpolytopes} to see that
\begin{align*}
 \psi_P(\gamma_k) &= \langle \mu_k, \gamma_k \rangle , & \psi_P(\ogamma_k) &= \langle \omu_k, \ogamma_k \rangle, &
 \psi_P(-\gamma_k) &= \langle \omu^k, - \gamma_k \rangle , & \psi_P(-\ogamma_k) & = \langle \mu^k, -\ogamma_k \rangle
\end{align*}
so that $P \subset \{ x \in \t_\R: \langle x, \gamma \rangle \leq M_\gamma, \langle x , - \gamma^* \rangle  \leq M^\gamma, \, \forall \gamma \in \Gamma \}$. 

To prove the other inclusion, consider the affine Weyl fan (see Figure \ref{figure:affineweylfan}. This fan has cones given by
\begin{align}
\begin{split}
C_{w_k} &= \text{span}_{\R_{\geq 0}} \{ \gamma_k, \gamma_{k+1} \}, \\
C^{w_k} &= \text{span}_{\R_{\geq 0}} \{ -\ogamma_k, -\ogamma_{k+1},\} 
\end{split} &
\begin{split}
C_{\ow_k} &= \text{span}_{\R_{\geq 0}} \{ \ogamma_k, \ogamma_{k+1} \}, \\
C_{\infty} &= \text{span}_{\R_{\geq 0}} \{ \omega_0 - \omega_1\},
\end{split} &
\begin{split}
C^{\ow_k} &= \text{span}_{\R_{\geq 0}} \{ -\gamma_k, -\gamma_{k+1} \}, \\
C_{\overline{\infty}} &= \text{span}_{\R_{\geq 0}} \{ -\omega_0 +\omega_1\}.
\end{split}\label{equation:affineweylfan}
\end{align}

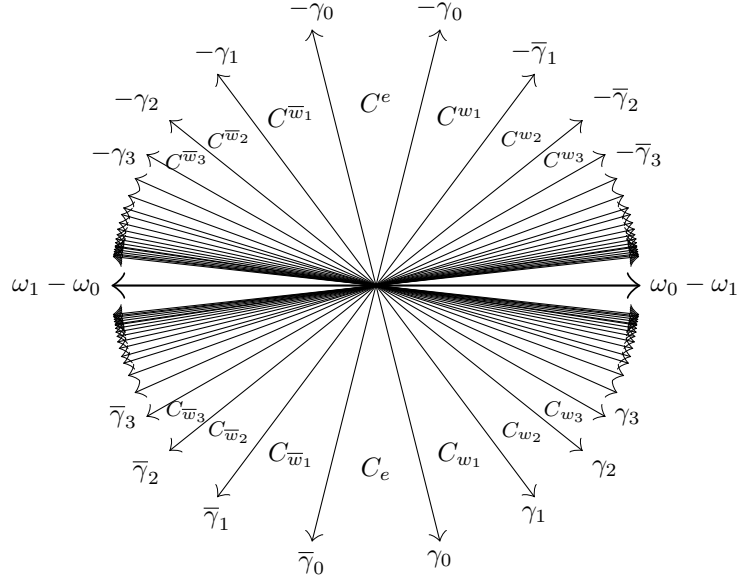
\begin{figure}
\centering
$
\begin{tikzpicture}

\draw[-{{Classical TikZ Rightarrow}[scale=2]}] (0,0) -- (3.5*.15/.62, -3.5*0.6/.62) node[below]{$\gamma_0$};
\draw[-{{Classical TikZ Rightarrow}[scale=2]}] (0,0) -- (3.5*.45/.75, -3.5*0.6/.75) node[below]{$\gamma_1$};
\draw[-{{Classical TikZ Rightarrow}[scale=2]}] (0,0) -- (3.5*.75/.96, -3.5*0.6/.96) node[below right]{$\gamma_2$};
\draw[-{{Classical TikZ Rightarrow}[scale=2]}] (0,0) -- (3.5*1.05/1.21, -3.5*0.6/1.21) node[right]{$\gamma_3$};
\draw[-{{Classical TikZ Rightarrow}[scale=2]}] (0,0) -- (3.5*1.35/1.48, -3.5*0.6/1.48) ;
\draw[-{{Classical TikZ Rightarrow}[scale=2]}] (0,0) -- (3.5*1.65/1.76, -3.5*0.6/1.76) ;
\draw[-{{Classical TikZ Rightarrow}[scale=2]}] (0,0) -- (3.5*1.95/2.04, -3.5*0.6/2.04) ;
\draw[-{{Classical TikZ Rightarrow}[scale=2]}] (0,0) -- (3.5*2.15/2.23, -3.5*0.6/2.23) ;
\draw[-{{Classical TikZ Rightarrow}[scale=2]}] (0,0) -- (3.5*2.45/2.52, -3.5*0.6/2.52) ;
\draw[-{{Classical TikZ Rightarrow}[scale=2]}] (0,0) -- (3.5*2.75/2.81, -3.5*0.6/2.81) ;
\draw[-{{Classical TikZ Rightarrow}[scale=2]}] (0,0) -- (3.5*3.05/3.11, -3.5*0.6/3.11) ;
\draw[-{{Classical TikZ Rightarrow}[scale=2]}] (0,0) -- (3.5*3.35/3.4, -3.5*0.6/3.4) ;

\node at (0, -2.5*0.6/.62) {$C_e$};
\node at (2.5*.3/.67, -2.5*.6/.67) {$C_{w_1}$};
\node at (2.75*.6/.85, -2.75*.6/.85) {{\footnotesize $C_{w_2}$}};
\node at (3*.9/1.08, -3*0.6/1.08) {{\footnotesize $C_{w_3}$}};

\node at (-2.5*.3/.67, -2.5*.6/.67) {$C_{\ow_1}$};
\node at (-2.75*.6/.85, -2.75*.6/.85) {{\footnotesize $C_{\ow_2}$}};
\node at (-3*.9/1.08, -3*0.6/1.08) {{\footnotesize $C_{\ow_3}$}};

\draw[-{{Classical TikZ Rightarrow}[scale=2]}] (0,0) -- (-3.5*.15/.62, -3.5*0.6/.62) node[below]{$\ogamma_0$};
\draw[-{{Classical TikZ Rightarrow}[scale=2]}] (0,0) -- (-3.5*.45/.75, -3.5*0.6/.75) node[below]{$\ogamma_1$};
\draw[-{{Classical TikZ Rightarrow}[scale=2]}] (0,0) -- (-3.5*.75/.96, -3.5*0.6/.96) node[below left]{$\ogamma_2$};
\draw[-{{Classical TikZ Rightarrow}[scale=2]}] (0,0) -- (-3.5*1.05/1.21, -3.5*0.6/1.21) node[left]{$\ogamma_3$};
\draw[-{{Classical TikZ Rightarrow}[scale=2]}] (0,0) -- (-3.5*1.35/1.48, -3.5*0.6/1.48) ;
\draw[-{{Classical TikZ Rightarrow}[scale=2]}] (0,0) -- (-3.5*1.65/1.76, -3.5*0.6/1.76) ;
\draw[-{{Classical TikZ Rightarrow}[scale=2]}] (0,0) -- (-3.5*1.95/2.04, -3.5*0.6/2.04) ;
\draw[-{{Classical TikZ Rightarrow}[scale=2]}] (0,0) -- (-3.5*2.15/2.23, -3.5*0.6/2.23) ;
\draw[-{{Classical TikZ Rightarrow}[scale=2]}] (0,0) -- (-3.5*2.45/2.52, -3.5*0.6/2.52) ;
\draw[-{{Classical TikZ Rightarrow}[scale=2]}] (0,0) -- (-3.5*2.75/2.81, -3.5*0.6/2.81) ;
\draw[-{{Classical TikZ Rightarrow}[scale=2]}] (0,0) -- (-3.5*3.05/3.11, -3.5*0.6/3.11) ;
\draw[-{{Classical TikZ Rightarrow}[scale=2]}] (0,0) -- (-3.5*3.35/3.4, -3.5*0.6/3.4) ;

\draw[-{{Classical TikZ Rightarrow}[scale=2]}] (0,0) -- (3.5*.15/.62, 3.5*0.6/.62) node[above]{$-\ogamma_0$};
\draw[-{{Classical TikZ Rightarrow}[scale=2]}] (0,0) -- (3.5*.45/.75, 3.5*0.6/.75) node[above]{$-\ogamma_1$};
\draw[-{{Classical TikZ Rightarrow}[scale=2]}] (0,0) -- (3.5*.75/.96, 3.5*0.6/.96) node[above right]{$-\ogamma_2$};
\draw[-{{Classical TikZ Rightarrow}[scale=2]}] (0,0) -- (3.5*1.05/1.21, 3.5*0.6/1.21) node[right]{$-\ogamma_3$};
\draw[-{{Classical TikZ Rightarrow}[scale=2]}] (0,0) -- (3.5*1.35/1.48, 3.5*0.6/1.48) ;
\draw[-{{Classical TikZ Rightarrow}[scale=2]}] (0,0) -- (3.5*1.65/1.76, 3.5*0.6/1.76) ;
\draw[-{{Classical TikZ Rightarrow}[scale=2]}] (0,0) -- (3.5*1.95/2.04, 3.5*0.6/2.04) ;
\draw[-{{Classical TikZ Rightarrow}[scale=2]}] (0,0) -- (3.5*2.15/2.23, 3.5*0.6/2.23) ;
\draw[-{{Classical TikZ Rightarrow}[scale=2]}] (0,0) -- (3.5*2.45/2.52, 3.5*0.6/2.52) ;
\draw[-{{Classical TikZ Rightarrow}[scale=2]}] (0,0) -- (3.5*2.75/2.81, 3.5*0.6/2.81) ;
\draw[-{{Classical TikZ Rightarrow}[scale=2]}] (0,0) -- (3.5*3.05/3.11, 3.5*0.6/3.11) ;
\draw[-{{Classical TikZ Rightarrow}[scale=2]}] (0,0) -- (3.5*3.35/3.4, 3.5*0.6/3.4) ;

\node at (0, 2.5*0.6/.62) {$C^e$};

\node at (2.5*.3/.67, 2.5*.6/.67) {$C^{w_1}$};
\node at (2.75*.6/.85, 2.75*.6/.85) {{\footnotesize $C^{w_2}$}};
\node at (3*.9/1.08, 3*0.6/1.08) {{\footnotesize $C^{w_3}$}};

\node at (-2.5*.3/.67, 2.5*.6/.67) {$C^{\ow_1}$};
\node at (-2.75*.6/.85, 2.75*.6/.85) {{\footnotesize $C^{\ow_2}$}};
\node at (-3*.9/1.08, 3*0.6/1.08) {{\footnotesize $C^{\ow_3}$}};

\draw[-{{Classical TikZ Rightarrow}[scale=2]}] (0,0) -- (-3.5*.15/.62, 3.5*0.6/.62) node[above]{$-\gamma_0$};
\draw[-{{Classical TikZ Rightarrow}[scale=2]}] (0,0) -- (-3.5*.45/.75, 3.5*0.6/.75) node[above]{$-\gamma_1$};
\draw[-{{Classical TikZ Rightarrow}[scale=2]}] (0,0) -- (-3.5*.75/.96, 3.5*0.6/.96) node[above left]{$-\gamma_2$};
\draw[-{{Classical TikZ Rightarrow}[scale=2]}] (0,0) -- (-3.5*1.05/1.21, 3.5*0.6/1.21) node[left]{$-\gamma_3$};
\draw[-{{Classical TikZ Rightarrow}[scale=2]}] (0,0) -- (-3.5*1.35/1.48, 3.5*0.6/1.48) ;
\draw[-{{Classical TikZ Rightarrow}[scale=2]}] (0,0) -- (-3.5*1.65/1.76, 3.5*0.6/1.76) ;
\draw[-{{Classical TikZ Rightarrow}[scale=2]}] (0,0) -- (-3.5*1.95/2.04, 3.5*0.6/2.04) ;
\draw[-{{Classical TikZ Rightarrow}[scale=2]}] (0,0) -- (-3.5*2.15/2.23, 3.5*0.6/2.23) ;
\draw[-{{Classical TikZ Rightarrow}[scale=2]}] (0,0) -- (-3.5*2.45/2.52, 3.5*0.6/2.52) ;
\draw[-{{Classical TikZ Rightarrow}[scale=2]}] (0,0) -- (-3.5*2.75/2.81, 3.5*0.6/2.81) ;
\draw[-{{Classical TikZ Rightarrow}[scale=2]}] (0,0) -- (-3.5*3.05/3.11, 3.5*0.6/3.11) ;
\draw[-{{Classical TikZ Rightarrow}[scale=2]}] (0,0) -- (-3.5*3.35/3.4, 3.5*0.6/3.4) ;

\draw[thick, -{{Classical TikZ Rightarrow}[scale=2]}] (0,0) -- (-3.5, 0) node[left]{$\omega_1 - \omega_0$};
\draw[thick, -{{Classical TikZ Rightarrow}[scale=2]}] (0,0) -- (3.5, 0) node[right]{$\omega_0 - \omega_1$};

\draw[-{{Classical TikZ Rightarrow}[scale=2]}] (0,0) -- (3.5*3.65/3.7, 3.5*0.6/3.7) ;
\draw[-{{Classical TikZ Rightarrow}[scale=2]}] (0,0) -- (3.5*3.95/4, 3.5*0.6/4) ;
\draw[-{{Classical TikZ Rightarrow}[scale=2]}] (0,0) -- (3.5*4.05/4.1, 3.5*0.6/4.1) ;
\draw[-{{Classical TikZ Rightarrow}[scale=2]}] (0,0) -- (3.5*4.35/4.39, 3.5*0.6/4.39) ;
\draw[-{{Classical TikZ Rightarrow}[scale=2]}] (0,0) -- (3.5*4.65/4.69, 3.5*0.6/4.69) ;
\draw[-{{Classical TikZ Rightarrow}[scale=2]}] (0,0) -- (3.5*4.95/4.99, 3.5*0.6/4.99) ;
\draw[-{{Classical TikZ Rightarrow}[scale=2]}] (0,0) -- (3.5*5.15/5.18, 3.5*0.6/5.18) ;
\draw[-{{Classical TikZ Rightarrow}[scale=2]}] (0,0) -- (3.5*5.45/5.48, 3.5*0.6/5.48) ;

\draw[-{{Classical TikZ Rightarrow}[scale=2]}] (0,0) -- (-3.5*3.65/3.7, -3.5*0.6/3.7) ;
\draw[-{{Classical TikZ Rightarrow}[scale=2]}] (0,0) -- (-3.5*3.95/4, -3.5*0.6/4) ;
\draw[-{{Classical TikZ Rightarrow}[scale=2]}] (0,0) -- (-3.5*4.05/4.1, -3.5*0.6/4.1) ;
\draw[-{{Classical TikZ Rightarrow}[scale=2]}] (0,0) -- (-3.5*4.35/4.39, -3.5*0.6/4.39) ;
\draw[-{{Classical TikZ Rightarrow}[scale=2]}] (0,0) -- (-3.5*4.65/4.69, -3.5*0.6/4.69) ;
\draw[-{{Classical TikZ Rightarrow}[scale=2]}] (0,0) -- (-3.5*4.95/4.99, -3.5*0.6/4.99) ;
\draw[-{{Classical TikZ Rightarrow}[scale=2]}] (0,0) -- (-3.5*5.15/5.18, -3.5*0.6/5.18) ;
\draw[-{{Classical TikZ Rightarrow}[scale=2]}] (0,0) -- (-3.5*5.45/5.48, -3.5*0.6/5.48) ;

\draw[-{{Classical TikZ Rightarrow}[scale=2]}] (0,0) -- (-3.5*3.65/3.7, 3.5*0.6/3.7) ;
\draw[-{{Classical TikZ Rightarrow}[scale=2]}] (0,0) -- (-3.5*3.95/4, 3.5*0.6/4) ;
\draw[-{{Classical TikZ Rightarrow}[scale=2]}] (0,0) -- (-3.5*4.05/4.1, 3.5*0.6/4.1) ;
\draw[-{{Classical TikZ Rightarrow}[scale=2]}] (0,0) -- (-3.5*4.35/4.39, 3.5*0.6/4.39) ;
\draw[-{{Classical TikZ Rightarrow}[scale=2]}] (0,0) -- (-3.5*4.65/4.69, 3.5*0.6/4.69) ;
\draw[-{{Classical TikZ Rightarrow}[scale=2]}] (0,0) -- (-3.5*4.95/4.99, 3.5*0.6/4.99) ;
\draw[-{{Classical TikZ Rightarrow}[scale=2]}] (0,0) -- (-3.5*5.15/5.18, 3.5*0.6/5.18) ;
\draw[-{{Classical TikZ Rightarrow}[scale=2]}] (0,0) -- (-3.5*5.45/5.48, 3.5*0.6/5.48) ;

\draw[-{{Classical TikZ Rightarrow}[scale=2]}] (0,0) -- (3.5*3.65/3.7, -3.5*0.6/3.7) ;
\draw[-{{Classical TikZ Rightarrow}[scale=2]}] (0,0) -- (3.5*3.95/4, -3.5*0.6/4) ;
\draw[-{{Classical TikZ Rightarrow}[scale=2]}] (0,0) -- (3.5*4.05/4.1, -3.5*0.6/4.1) ;
\draw[-{{Classical TikZ Rightarrow}[scale=2]}] (0,0) -- (3.5*4.35/4.39, -3.5*0.6/4.39) ;
\draw[-{{Classical TikZ Rightarrow}[scale=2]}] (0,0) -- (3.5*4.65/4.69, -3.5*0.6/4.69) ;
\draw[-{{Classical TikZ Rightarrow}[scale=2]}] (0,0) -- (3.5*4.95/4.99, -3.5*0.6/4.99) ;
\draw[-{{Classical TikZ Rightarrow}[scale=2]}] (0,0) -- (3.5*5.15/5.18, -3.5*0.6/5.18) ;
\draw[-{{Classical TikZ Rightarrow}[scale=2]}] (0,0) -- (3.5*5.45/5.48, -3.5*0.6/5.48) ;

\end{tikzpicture}
$
\caption{The affine Weyl fan}
\label{figure:affineweylfan}
\end{figure}

The union of these cones is $\t_\R^*$ and $\psi_P$ is linear on each of these cones \cite{AffineMVpolytopes}. Consider $x$ such that $\langle x, \gamma \rangle \leq M_\gamma, \langle x, - \gamma^* \rangle \leq M^\gamma, \, \forall \gamma \in \Gamma$. If $\theta \in C_{w_k}$, then $\theta = n_1 \gamma_k + n_2 \gamma_{k+1}$. By linearity of $\psi_P$ and the definition of $M_\gamma$, $\psi_P(\theta) = n_1 M_{k} + n_2 M_{k+1}$ so that
\begin{align*}
\langle x, \theta \rangle & = n_1 \langle  x, \gamma_k \rangle + n_2 \langle x, \gamma_{k+1} \rangle  \leq n_1 M_k + n_2 M_{k+1} = \psi_P(\theta).
\end{align*}
A similar argument works for $\theta \in C_{\ow_k}\cup C^{w_k} \cup C^{\ow_k}$. 

For $\theta \in C_\infty$, then $\theta = n_1 (\omega_0 - \omega_1)$ for some $n \in \R_{\geq0}$. By linearity of $\phi_P$ and the definition of $M_\gamma$, $\psi_P(\theta) = n M_1 + n M^1$ so that
\begin{align*}
 \langle x, \theta \rangle &= n \langle x, \omega_0 \rangle + n \langle x, - \omega_1 \rangle \leq n M_1 + n M^1 = \psi_P(\theta).
\end{align*}
A similar argument works for $\theta \in C_{\overline{\infty}}$. Thus $P = \{ x \in \t_\R : \langle x, \gamma \rangle \leq M_\gamma, \langle x, - \gamma^* \rangle \leq M^\gamma, \, \forall \gamma \}$. 
\end{proof}

By the definition of the hyperplane data, the vertices are given by $\mu_k =  \sum_{i=0}^1 M_{w_k \cdot \omega_i} w_k \cdot \alpha^\vee_i$ and $\omu_k =\sum_{i=0}^1 M_{\ow_k \cdot \omega_i} \ow_k \cdot \alpha^\vee_i$, where $w_0 = \ow_0 = e$ the identity. Thus, for $k \geq 0$, the vertices are given by
\begin{align*}
\mu_k &= M_k (\alpha^\vee_1 + k\delta)- M_{k+1} (\alpha^\vee_1 + (k-1) \delta), &
\omu_k &= M_{\overline{k}} (\alpha^\vee_0 + k \delta) -M_{\overline{k+1}} (\alpha^\vee_0 + (k-1) \delta), \\
\mu^k &= M^k (\alpha^\vee_0 + k\delta)- M^{k+1} (\alpha^\vee_0 + (k-1) \delta), &
\omu^k &= M^{\overline{k}} (\alpha^\vee_1 + k \delta) - M^{\overline{k+1}} (\alpha^\vee_1 + (k-1) \delta). 
\end{align*}
Note we use the notation $w_0=e$ for convenience when we are listing the elements $\{w_k: k \in \N\}$. This is not to be confused the notation for the longest Weyl element $w_0$ in the finite case.

The relation between the Lusztig data and the hyperplane data is given by
\begin{align*}
a_k &= 2M_k- M_{k-1} - M_{k+1}, &
\bar{a}_k &= 2M_{\overline{k}} - M_{\overline{k-1}} - M_{\overline{k+1}},\\
a^k &= -2M^k+ M^{k-1} + M^{k+1}, &
\bar{a}^k &= -2M^{\overline{k}} + M^{\overline{k-1}} + M^{\overline{k+1}}.
\end{align*}

The diagonal relations can be rewritten in terms of the hyperplane data of the polytope. 

\begin{proposition}\label{proposition:diagonalrelations} 

Let $P$ be a GGMS polytope which has lower hyperplane data $(M_\gamma)_{\gamma \in \Gamma}$ and upper hyperplane data $(M^\gamma)_{\gamma \in \Gamma}$. 
The condition \ref{condition:lower} in Definition \ref{definition:affinepolytope} on the vertices of $P$ is equivalent to equations:
\begin{align*}
k M_k+k M_{\overline{k}} = \min \lbrace &   k M_{k-1} + M_k + (k-1) M_{\overline{k+1}}, k M_{\overline{k-1}} + M_{\overline{k}} + (k-1) M_{k+1} \rbrace, \text{ for } k \geq 2.
\end{align*}

The condition \ref{condition:upper} in Definition \ref{definition:affinepolytope} on the vertices of $P$ is equivalent to the equations:
\begin{align*}
 k M^k + k M^{\overline{k}} = \max \lbrace & k M^{k-1} + M^k + (k-1) M^{\overline{k+1}} ,k M^{\overline{k-1}} + M^{\overline{k}} + (k-1) M^{k+1}\rbrace, \text{ for } k \geq 2.
\end{align*}
\end{proposition}

\begin{proof}
Let $k \geq 2$. First, notice that
\begin{align*}
\langle \omu_k - \mu_{k-1}, \omega_1 \rangle
& = k M_{\overline{k}} -(k-1)M_{\overline{k+1}}- k M_{k-1} + (k-1)M_k, \\
\langle \mu_k - \omu_{k-1}, \omega_0 \rangle
 & = k M_k - (k-1) M_{k+1} -  kM_{\overline{k-1}}  +(k-1)M_{\overline{k}}.
\end{align*}

The maximum of $\{ \langle \omu_k - \mu_{k-1}, \omega_1 \rangle , \langle  \mu_k - \omu_{k-1}, \omega_0 \rangle \} $ is equal to the minimum
\[
-\min\{ (k-1)M_{\overline{k+1}}+  k M_{k-1} + M_k , (k-1) M_{k+1} + kM_{\overline{k-1}} + M_{\overline{k}} \} +k M_k+  k M_{\overline{k}}
\]
so that the condition $\max\{\langle \omu_k - \mu_{k-1}, \omega_1 \rangle, \langle \mu_k - \omu_{k-1}, \omega_0 \rangle\} =0$ is equivalent to 
\[
\min\{ (k-1)M_{\overline{k+1}}+  k M_{k-1} + M_k , (k-1) M_{k+1} + kM_{\overline{k-1}} + M_{\overline{k}} \} =k M_k+  k M_{\overline{k}}.
\]

A similar proof works for the upper diagonal equations.
\end{proof}

When $P$ is an MV polytope, the hyperplane data will satisfy the diagonal relations. We will call the lower hyperplane data $(M_\gamma)_{\gamma \in \Gamma}$ a \emph{lower BZ datum} when the collection satisfies the lower diagonal relations \ref{condition:lower} in Definition \ref{definition:affinepolytope}. We call the upper hyperplane data $(M^\gamma)_{\gamma \in \Gamma}$ an \emph{upper BZ datum} when the collection satisfies the upper diagonal relations \ref{condition:upper} in Definition \ref{definition:affinepolytope}. We call $(M_\gamma, M^\gamma)_{\gamma \in \Gamma}$ a \emph{BZ datum} when $(M_\gamma)_{\gamma \in \Gamma}$ is a lower BZ datum and $(M^\gamma)_{\gamma \in \Gamma}$ is an upper BZ datum. 

Similar to the finite case, the Lusztig data along one minimal path completely determines an MV polytope and every possible Lusztig data results in an MV polytope.

\begin{theorem}[{\cite[Theorem 3.11]{Rank2affine}}]
Consider a collection $(a_n, \lambda_n, a^n)_{n \in \N}$ such that  
\begin{enumerate}[label=(\roman*)]
 \item $(a_n, \lambda_n, a^n) \in \Z^3$ for all $n$
 \item for large enough $k$, $a_k = \lambda_k = a^k =0$
 \item $\lambda_1 \geq \lambda_2 \geq \cdots$
\end{enumerate}
Then there is a unique affine MV polytope $P$ whose right Lusztig data are $(a_n, \lambda_n, a^n)_{n \in \N}$. 
\end{theorem}
In the proof of this theorem, the left Lusztig data are explicitly constructed from the right Lusztig data. Moreover, by \cite[Remark 3.22]{Rank2affine}, the resulting polytope $P$ with right Lusztig data $(a_i, \lambda, a^i)$ has left Lusztig data $(\bar{a}_i, \olambda, \bar{a}^i)$ where
\begin{align*}
\bar{a}_1 = \max \{& (k-1) a_k + (k-2) a_{k-1} - 2a_{k-2} - \cdots - 2a_2 - 2a_1, \text{ for } k \geq 2, \\
& \lambda_1- \cdots - 2a_3 - 2a_2 - 2a_1, \\
&ka^k + (k+1) a^{k+1} + 2a^{k+2} + 2a^{k+3} + \cdots - 2a_{k-2} - \cdots - 2a_2 - 2a_1, \text{ for } k \geq 1\}.
\end{align*}

We will give an alternate proof of this equality for lower affine MV polytopes using the diagonal relations. To write the diagonals in terms of the Lusztig data, note that for $k \geq 2$,
\begin{align*}
\langle \omu_k - \mu_{k-1}, \omega_1 \rangle & = \sum_{s=1}^k \langle \omu_s - \omu_{s-1}, \omega_1 \rangle - \sum_{s=1}^{k-1}\langle \mu_s -\mu_{s-1}, \omega_1 \rangle = \sum_{s=2}^k (s-1)\bar{a}_s - \sum_{s=1}^{k-1}sa_s, \\
\langle \mu_k - \omu_{k-1}, \omega_0 \rangle & = \sum_{s=1}^k \langle \mu_s - \mu_{s-1}, \omega_0 \rangle - \sum_{s=1}^{k-1} \langle  \omu_s -\omu_{s-1}, \omega_0 \rangle = \sum_{s=2}^k (s-1) a_s - \sum_{s=1}^{k-1}s\bar{a}_s.
\end{align*}
Thus the lower diagonal equations can be written as
\begin{align*}
\max \{ &(k-1) \bar{a}_k + (k-2) \bar{a}_{k-1} + \cdots + 2\bar{a}_3 + \bar{a}_2 - a_1 - 2 a_2 - \cdots - (k-1) a_{k-1},  \\  & (k-1) a_k + (k-2) a_{k-1} + \cdots + 2 a_3 +  a_2 -\bar{a}_1 - 2\bar{a}_2 - \cdots - (k-1)\bar{a}_{k-1} \} =0.
\end{align*}

\begin{lemma}
For an affine MV polytope $P$ with right Lusztig data $(a_i, \lambda, a^i)$, 
\begin{align}\label{equation:bara1}
\bar{a}_1 \geq \max_{\substack{n\in\N \\ n \geq 2}}\{ (n-1) a_n + (n-2)a_{n-1} - 2a_{n-2} - \cdots - 2 a_2 - 2a_1\}.
\end{align}
If the right Lusztig data of the polytope has $\lambda =0$ and $a^k = 0$ for all $k \geq 0$, then this is an equality. 
\end{lemma}

\begin{proof}
Let $k \geq 2$. Consider the second term in the maximum of the $k+1$ diagonal. As this term is at most 0, we have
\begin{align*}
\bar{a}_1 &\geq k a_{k+1} + (k-1) a_{k} + \cdots + 2a_3 + a_2 - 2\bar{a}_2 - \cdots - k \bar{a}_{k}.
\end{align*}
Now consider the $k^\text{th}$ diagonal. As the first term is at most 0, then isolating for $\bar{a}_k$, we have
\begin{align}\label{equation:secondaffinediagonal}
\bar{a}_{k} &\leq \frac{1}{k-1} \left( a_1 + 2a_2 + \cdots + (k-1) a_{k-1} - (k-2) \bar{a}_{k-1} - \cdots - 2 \bar{a}_3 - \bar{a}_2 \right)
\end{align}
Then 
\begin{align*}
\bar{a}_1 \geq k a_{k+1} + (k-1) a_{k} + \sum_{i=1}^{k-1} \left((i-1) - \frac{ik}{k-1} \right) a_{i}  -\sum_{i=2}^{k-1} \left(i - \frac{(i-1)k}{k-1} \right) \bar{a}_i.
\end{align*} By iterating this technique of removing the largest $\bar{a}_{k-n}$ by using the $k-n$ diagonal to find $\bar{a}_{k-n} \leq \frac{1}{k-n-1} \left( \sum_{i=1}^{k-n-1} i a_i - \sum_{i=2}^{k-n-1} (i-1)\bar{a}_i \right)$, this simplifies to
\[
\bar{a}_1 \geq ka_{k+1} + (k-1) a_k  -2\sum_{i=1}^{k-1}a_{i}.
\]
So we have shown $\bar{a}_1$ is larger than the maximum in (\ref{equation:bara1}). 

Suppose the right Lusztig data are $(a_n, 0,0)_{n \in \N}$. From the diagonal relations, either there exists $k \in \N$ such that $\langle \mu_k - \omu_{k-1}, \omega_0 \rangle=0$ or $\langle \omu_k - \mu_{k-1}, \omega_1 \rangle =0$ for every $k$. 

Suppose we are in the first case and $n$ is the smallest number such that $\langle \mu_n - \omu_{n-1}, \omega_0 \rangle=0$. Using induction, it is easy to see that $\bar{a}_k = a_{k+1}$ for all $k \leq n$. By assumption, the second term in the $n^\text{th}$ diagonal is zero and hence $\bar{a}_1 = (n-1)a_n + (n-2) a_{n-1} - 2a_{n-2} - \cdots -2 a_2 - 2a_1$. 

Suppose we are in the second case so that $\langle \omu_k - \mu_{k-1}, \omega_1 \rangle=0$ for every $k$ and hence $\bar{a}_k = a_{k+1}$ for every $k$. As these polytopes are finite, there exists an $n$ such that $a_n \neq 0$ but $a_{n+i} =0$ for every $i \geq 1$. i.e. $\mu_n$ is the highest vertex on the right side of $L(P)$. It follows that $\bar{a}_{n-1} =a_{n} \neq 0$ but $\bar{a}_{n+i}=0$ for all $i \geq 0$ so $\omu_{n+i} = \omu_{n-1}$ for all $i \geq 0$. Thus $\omu_{n-1}$ is the highest vertex on the left side of $L(P)$. 

Now, consider the upper left vertices. Let $i$ be the smallest number such that $\omu^i \neq \mu_n$ (as the right Lusztig data are $(a_n, 0,0)_{n \in \N}$, $\omu^0 = \mu_n$ so $i >0$). Then $\mu_n - \omu^i = \bar{a}^i (\alpha^\vee_1 + (i-1) \delta)$ and thus $\langle \mu_n - \omu^i, \omega_0 \rangle \geq 0$. 

Let $j$ be the largest number such that $\omu^j \neq \omu_{n-1}$. Then $\omu^j - \omu_{n+1} = \bar{a}^{j+1}(\alpha^\vee_1 + j \delta)$ and $ \langle \omu^j - \omu_{n-1}, \omega_0 \rangle  \geq 0$. Similarly, for any $k$, $\omu^{k-1} - \omu^k = \bar{a}^k (\alpha^\vee_1 + (k-1) \delta)$ so $\langle \omu^{k-1} - \omu^k , \omega_0 \rangle \geq 0$. It follows that
\begin{align*}
\langle \mu_n - \omu_{n-1}, \omega_0 \rangle = \left\langle \mu_n -\omu^i + \sum_{k=i}^{j-1} \left( \omu^k - \omu^{k+1}\right) + \omu^{j} - \omu_{n-1},  \omega_0 \right\rangle  \geq 0.
\end{align*}
As the $n^\text{th}$ diagonal relation guarantees $\langle \mu_n - \omu_{n-1}, \omega_0 \rangle \leq 0$, then necessarily $\langle \mu_n - \omu_{n-1}, \omega_0 \rangle =0$. This means that only the first case is possible and so  $\bar{a}_1 = (m-1)a_m + (m-2) a_{m-1} - 2a_{m-2} - \cdots 2 a_2 - 2a_1$ for the smallest $m$ such that $\langle \mu_m - \omu_{m-1}, \omega_0 \rangle = 0$. 
\end{proof}

\subsection{Upper and lower polytopes}\label{section:highestvertexpoly}

In this section, for $w \in W$, we will use the shorthand 
\begin{align*}
\mu_w : = \begin{cases} \mu_k, & \text{ if } w = w_k \\
\omu_k, & \text{ if } w = \ow_k \end{cases} && \mu^w : = \begin{cases} \omu^k, & \text{ if } w = w_k \\
\mu^k, & \text{ if } w = \ow_k \end{cases} 
\end{align*}

\begin{definition}
A \emph{lower affine MV polytope of highest vertex at most $w$} is an affine MV polytope such that $\mu_w = \mu^0$. Denote the set of lower affine MV polytopes of highest vertex $w$ by $\Pw$. 
\end{definition}

\begin{definition}
An \emph{upper affine MV polytope of lowest vertex $w$} is an affine MV polytope such that $\mu^w = \mu_0$. 
\end{definition}

As upper and lower polytopes are closely related, we will only consider lower polytopes for the rest of the paper. First, the condition $\mu_w= \mu^0$ will induce some conditions on the vertices on the other side of the polytope. 

\begin{lemma}\label{lemma:mu_m-1=mu^1}
Consider $P \in \Pw$. If $w = w_m$ for some $m \in \N$, then $\omu_{m-1} = \omu^1$. If $w = \ow_m$ for $m \in \N$, then $\mu_{m-1} = \mu^1$. 
\end{lemma}

\begin{proof}
Without loss of generality, suppose that $w = w_m$. First, notice that 
\begin{align*}
\mu^0 - \omu_{m-1} &= \sum_{s=1}^\infty (\omu^{s-1} - \omu^s) + (\omu^\infty - \omu_\infty) + \sum_{s=m-1}^\infty (\omu_{s+1} - \omu_{s}) \\
& = \sum_{s=1}^\infty (\bar{a}^s (\alpha^\vee_1 + (s-1)\delta)) + |\overline{\lambda}| \delta + \sum_{s=m-1}^\infty (\bar{a}_{s+1}(\alpha^\vee_0 + s\delta)).
\end{align*}
Thus, $\langle \mu^0 - \omu_{m-1}, \omega_0 \rangle = \sum_{s=1}^\infty (s-1) \bar{a}^s + |\delta| + \sum_{s=k}^\infty (s+1) \bar{a}_{s+1} \geq 0$ as each term is non-negative. But by the lower diagonals \ref{condition:lower} in Definition \ref{definition:affinepolytope}, 
\[
\langle \mu^0 - \omu_{m-1}, \omega_0 \rangle = \langle \mu_m - \omu_{m-1}, \omega_0 \rangle \leq 0.
\]
Thus $\langle \mu^0 - \omu_{m-1}, \omega_0 \rangle= \sum_{s=1}^\infty (s-1) \bar{a}^s + |\delta| + \sum_{s=m-1}^\infty (s+1) \bar{a}_{s+1} =0$ and hence $\bar{a}^s=0$ for $s \geq 2$, $|\delta|=0$ and $\bar{a}_{t}=0$ for $t \geq m-1$. It follows that $\omu_{m-1} = \omu^1$. 
\end{proof}

As in the finite case, we can consider the dual fan of $P \in \Pw$. Recall the affine Weyl fan defined by the maximal cones (\ref{equation:affineweylfan}).  Let $\ell(w) =m$. Define the cones
\begin{align*}
D_{m} &= \begin{cases} \left(\bigcup_{k \geq m} C_{w_k}\right) \cup C_{\infty} \cup \left(\bigcup_{k \in \N} C^{w_k}\right), & \text{ if } w=w_m \\
\left( \bigcup_{k \geq m} C_{\ow_k}\right) \cup C_{\overline{\infty}} \cup\left( \bigcup_{k \in \N} C^{\ow_k} \right), & \text{ if } w=\ow_m
 \end{cases}\\
 D_{m-1} &= \begin{cases} \left( \bigcup_{k \geq m-1} C_{\ow_k} \right) \cup C_{\overline{\infty}} \cup \left( \bigcup_{k \geq 1} C^{\ow_k} \right), & \text{ if } w=w_m \\
\left( \bigcup_{k \geq m-1} C_{w_k} \right) \cup C_\infty \cup\left( \bigcup_{k \geq 1} C^{w_k}\right), & \text{ if } w=\ow_m
 \end{cases}
\end{align*}
Note that $D_m$ will always contain the cone $C^e$, as well as all the cones labelled by the Weyl elements associated to the vertices that lie between $\mu_w$ and $\mu^0$ on a minimal path in $P$.  
Hence the dual fan $\mathcal{N}(P)$ is a coarsening of the fan generated by the maximal cones:
\begin{align*}
 D_{m}, && D_{m-1}, && C_{w_k} \text{ for } 1\leq k \leq m-1, && C_{\ow_k} \text{ for } 1 \leq k \leq m-2.
\end{align*}
For an example, see Figure \ref{figure:affinedualfan}. 

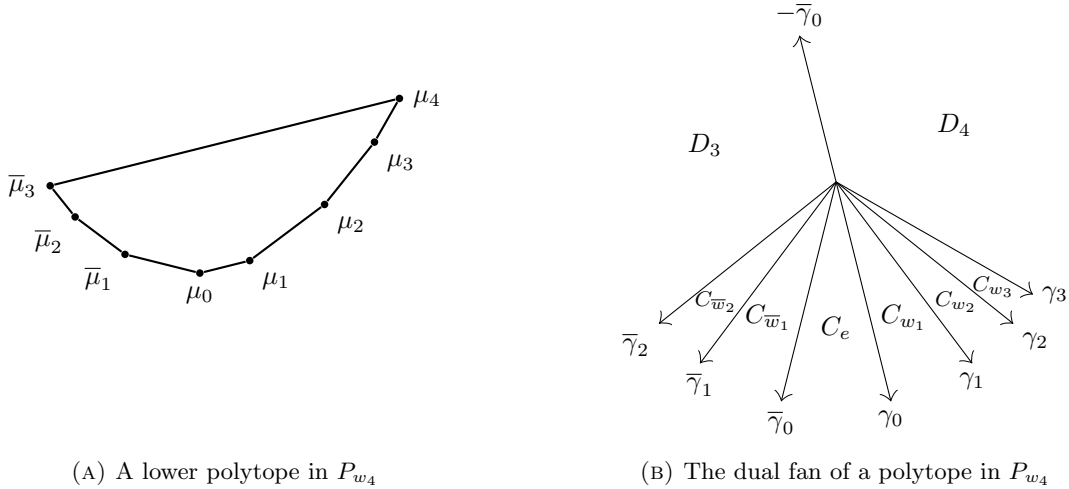
\begin{figure}[ht] 
\begin{subfigure}{.49\linewidth}
\centering
$
\begin{tikzpicture}[scale=.55]
\node[vertex] (O) at (0,0) [label=below:$\mu_0$] {};
\node[vertex] (U1) at (1.2,0.3)[label=below right:$\mu_1$] {};
\node[vertex] (U2) at (3, 1.65)[label=below right:$\mu_2$] {};
\node[vertex] (U3) at (4.2, 3.15)[label=below right:$\mu_3$] {};
\node[vertex] (U4) at (4.8, 4.2) [label=right:$\mu_4$] {};
\node[vertex] (W3) at (-3.6, 2.1) [label=left:$\omu_3$]{};
\node[vertex] (W2) at (-3, 1.35)[label=below left:$\omu_2$] {};
\node[vertex] (W1) at (-1.8, 0.45)[label=below left:$\omu_1$] {};

\draw[thick] (O) -- (U1) -- (U2) -- (U3) -- (U4)  -- (W3) -- (W2) -- (W1) -- (O);

\node at (4*.15/.62, -4*0.6/.62) {};
\node at (-7*.15/.62, 7*0.6/.62) {};

\end{tikzpicture}
$
\caption{A lower polytope in $P_{w_4}$}
\end{subfigure}
\begin{subfigure}{.49\linewidth}
\centering
$
\begin{tikzpicture}

\draw[-{{Classical TikZ Rightarrow}[scale=2]}] (0,0) -- (3*.15/.62, -3*0.6/.62) node[below]{$\gamma_0$};
\draw[-{{Classical TikZ Rightarrow}[scale=2]}] (0,0) -- (3*.45/.75, -3*0.6/.75) node[below]{$\gamma_1$};
\draw[-{{Classical TikZ Rightarrow}[scale=2]}] (0,0) -- (3*.75/.96, -3*0.6/.96) node[below right]{$\gamma_2$};
\draw[-{{Classical TikZ Rightarrow}[scale=2]}] (0,0) -- (3*1.05/1.21, -3*0.6/1.21) node[right]{$\gamma_3$};

\draw[-{{Classical TikZ Rightarrow}[scale=2]}] (0,0) -- (-3*.15/.62, -3*0.6/.62) node[below]{$\ogamma_0$};
\draw[-{{Classical TikZ Rightarrow}[scale=2]}] (0,0) -- (-3*.45/.75, -3*0.6/.75) node[below]{$\ogamma_1$};
\draw[-{{Classical TikZ Rightarrow}[scale=2]}] (0,0) -- (-3*.75/.96, -3*0.6/.96) node[below left]{$\ogamma_2$};

\draw[-{{Classical TikZ Rightarrow}[scale=2]}] (0,0) -- (-2*.15/.62, 2*0.6/.62) node[above]{$-\ogamma_0$};

\node at (0, -2*0.6/.62) {$C_e$};
\node at (2*.3/.67, -2*.6/.67) {$C_{w_1}$};
\node at (2.25*.6/.85, -2.25*.6/.85) {{\footnotesize $C_{w_2}$}};
\node at (2.5*.9/1.08, -2.5*0.6/1.08) {{\footnotesize $C_{w_3}$}};

\node at (-2*.3/.67, -2*.6/.67) {$C_{\ow_1}$};
\node at (-2.25*.6/.85, -2.25*.6/.85) {{\footnotesize $C_{\ow_2}$}};

\node at (-2*1.05/1.21, 2*0.6/1.21-.5) {$D_{3}$};

\node at (2*.75/.96, 2*0.6/.96-.5) {$D_{4}$};

\end{tikzpicture}
$
\caption{The dual fan of a polytope in $P_{w_4}$}
\end{subfigure}
\caption{A polytope in $\Pw$ and its dual fan}
\label{figure:affinedualfan}
\end{figure}

\begin{remark}
For $P$ an affine MV polytope, every lower MV polytope $L(P) = \text{conv}\{ \mu_k, \omu_k: k \in \N\}$ is a GGMS polytope by Lemma \ref{lemma:lumGGMS}. It is easy to see that $L(P)$ is in fact an affine MV polytope as $L(P)$ will inherit all the conditions in Definition \ref{definition:affinepolytope} from $P$. 

As an affine MV polytope, the vertex $\mu^0$ of $L(P)$ must be such that $\mu^0 = \mu_k$ or $\omu_k$ for some $k \in \N$. This $k$ must also be such that $\mu_\infty = \mu_k$ or $\omu_\infty = \omu_k$. Thus $L(P) \in \Pw$ for some $w \in W$. 
\end{remark}

Let $\ui = (i_1, \dots, i_m)$ be such that $w = s_{i_1} \dots s_{i_m}$ is a reduced word. 

\begin{lemma}\label{lemma:BZdata}
The collection $(M_\gamma)_{\gamma \in \Gamma}$ is the lower BZ datum of a lower affine MV polytope of highest vertex $w$ if each $M_\gamma \in \Z$ and:
\begin{enumerate}[label=(\roman*)]
 \item \label{condition:positivity}   For $1 \leq k \leq m$, $M_{\overline{k-1}} + M_{\overline{k+1}} \leq 2M_{\overline{k}}$ and $M_{k-1} + M_{k+1} \leq 2M_{k}$. 
 \item \label{condition:edgeequalities} For $k \geq m+1$, $M_{\overline{k-1}} + M_{\overline{k+1}} = 2M_{\overline{k}}$ and  $M_{k-1} + M_{k+1} = 2M_{k}$. 
 \item \label{condition:edgeequalities2} For $k=m$, if $i_1 =1$ then $M_{\overline{m-1}} + M_{\overline{m+1}} = 2 M_{\overline{m}}$. If $i_1 =0$, then $M_{m-1} + M_{m+1} =2 M_m$.
 \item \label{condition:diagonal} For $2 \leq k$, $ k M_{\overline{k}}+   k M_k = \min \lbrace  k M_{\overline{k-1}} + M_{\overline{k}} + (k-1) M_{k+1}, k M_{k-1} + M_{k} + (k-1) M_{\overline{k+1}}  \}$.
\end{enumerate}
Conversely, for an affine MV polytope $P$, the lower polytope $L(P)$ is a lower affine MV polytope of highest vertex $w$ only if the lower BZ data $(M_\gamma)_{\gamma \in \Gamma}$ satisfy the conditions \ref{condition:positivity} - \ref{condition:diagonal}. 
\end{lemma}

\begin{proof}
Without loss of generality, suppose $w = w_m$ so that $i_1=1$. 

Let $P$ be an affine MV polytope and suppose that $L(P)$ is a lower affine MV polytope of highest vertex $w$. Then $(M_\gamma)_{\gamma\in\Gamma}$ is a lower BZ datum and the conditions \ref{condition:positivity} and \ref{condition:diagonal} follow directly. We need to show that the edge equalities \ref{condition:edgeequalities} hold exactly when $(M_\gamma)_{\gamma \in \Gamma}$ is the lower BZ datum of a lower polytope of highest vertex $w$. 

By definition of $L(P)$, $\mu_w = \mu^0$, so that for every $k \geq m +1$, $\mu_k = \mu_w$. In particular, $\mu_{k} = \mu_{k-1}$ so that the Lusztig data $a_k =0$. Thus $M_{k-1} + M_{k+1} = 2M_k$ for $k \geq m+1$. 

If $k \geq m$, the edge equality $M_{\overline{k-1}} + M_{\overline{k+1}} = 2M_{\overline{k}}$ is equivalent to $\omu_{k-1} = \omu_{k}$. By Lemma \ref{lemma:mu_m-1=mu^1}, $\mu_w = \mu^0 \implies \omu_{m-1} = \omu^1$ and hence the rest of the edge equalities in \ref{condition:edgeequalities} hold. 

Suppose that $(M_\gamma)_{\gamma \in \Gamma}$ satisfies \ref{condition:positivity} - \ref{condition:diagonal}.  All we need to show is that the edge equalities in \ref{condition:edgeequalities} imply the diagonal relations for \ref{condition:diagonal} for $k>m$ so that $(M_\gamma)_{\gamma \in \Gamma}$ is a lower BZ datum. 

For $k\geq m$ suppose the diagonal relation holds: 
\begin{align*}
-\min \{ (k-1) M_{\overline{k+1}} + kM_{k-1} + M_k, (k-1) M_{k+1} + k M_{\overline{k-1}} + M_{\overline{k}} \} + k M_k + k M_{\overline{k}}=0.
\end{align*}

By adding $k(2M_{\overline{k+1}} -M_{\overline{k}} - M_{\overline{k+2}}) - k (2M_k - M_{k-1} - M_{k+1})$ to the first term in the diagonal, this term becomes
\begin{align*}
-kM_{\overline{k+2}}- (k+1)M_k - M_{k+1} + (k+1) M_{\overline{k+1}} +(k+1)M_{k+1}.
\end{align*}
Similarly, by adding $k(2M_{k+1} -M_k - M_{k+2}) - k (2M_{\overline{k}} - M_{\overline{k-1}} - M_{\overline{k+1}})$ to the second term, we get 
\[
 -k M_{k+2} -(k+1) M_{\overline{k}} - M_{\overline{k+1}}+ (k+1) M_{k+1} + (k+1) M_{\overline{k+1}}.
\]
Thus the $k+1^\text{st}$ diagonal equation also holds:
\begin{align*}
\min \{ kM_{\overline{k+2}} + (k+1)M_{k} + M_{k+1}, k M_{k+2} +(k+1) M_{\overline{k}} + M_{\overline{k+1}} \} = (k+1) M_{k+1} + (k+1) M_{\overline{k+1}}
\end{align*}
and by induction the diagonals hold for all $k$. 
\end{proof}

Condition \ref{condition:positivity} guarantees that the sides of the polytope are non-negative and conditions \ref{condition:edgeequalities}-\ref{condition:edgeequalities2} guarantee that the highest vertex is at most $\mu_w$. Condition \ref{condition:diagonal} guarantees the diagonal relations are satisfied. 

Let $\mathcal{M}^w_\Gamma$ be the set of all lower BZ data $(M_\gamma)_{\gamma\in \Gamma}$ that satisfy conditions \ref{condition:positivity}, \ref{condition:edgeequalities} and \ref{condition:diagonal}. Then we have a bijection between these BZ datum and polytopes of highest vertex $w$. 

\begin{theorem}\label{theorem:BZ}
There is a bijection $\Pw \rightarrow \mathcal{M}^w_\Gamma$ by $P \mapsto (M_\gamma)_{\gamma \in \Gamma}$ as defined in Definition \ref{definition:affineBZdata}. 
Hence an MV polytope of highest vertex $w$ is completely determined by its lower BZ data. 
\end{theorem}

\begin{proof}
By Lemma \ref{lemma:BZdata}, the map that sends $P$ to its lower BZ data is a map from $\Pw \rightarrow \mathcal{M}^w_\Gamma$. We need to show that a polytope $P \in \Pw$ is completely determined by its lower BZ data. 
 
Suppose $(M_\gamma)_{\gamma \in \Gamma}$ is a collection as in Lemma \ref{lemma:BZdata}. Without loss of generality, assume that $w = w_m$. Set
\begin{align*}
M^1 &= -(m+1)M_m + (m-1)M_{m+1},  & M^{\overline{1}} & = -m M_{m+1} - m M_m, \\
M^{\overline{2}} &= -(m+2) M_{m+1} - (m-2)M_m - M_{\overline{m}} + M_{\overline{m-1}}.
\end{align*}
We recursively define $M^{k+1} = 2M^{k} - M^{k-1}$ for $k \geq 1$ and $M^{\overline{k+1}} = 2M^{\overline{k}} - M^{\overline{k-1}}$ for $k \geq 2$. Thus we have a collection $(M^\gamma)_{\gamma \in \Gamma}$ of integers which satisfy the edge equalities and diagonal relations trivially.
 
Define the map from  $\mathcal{M}_\Gamma^w$ to convex polytopes where $(M_\gamma)_{\gamma \in \Gamma}$ is sent to the decorated polytope $P$ with BZ data $(M_\gamma, M^\gamma)_{\gamma \in \Gamma}$ and decoration $\lambda = \lambda' = 0$. We need to prove that this polytope is a lower affine MV polytope of highest vertex $w$. 

First, from the definition of $M^{\overline{2}}$, $(-2M^{\overline{1}} + M^{\overline{2}} + M^1) \alpha^\vee_1
= \mu_m - \omu_{m-1}$ so that $-2M^{\overline{1}} + M^{\overline{2}} + M^1 = \bar{a}^1 \in \Z$ and $P$ is a GGMS polytope. 

By Lemma \ref{lemma:BZdata}, $(M_\gamma)_{\gamma \in \Gamma}$ is the lower BZ data of a lower affine MV polytope of highest vertex $w$ and $(M^\gamma)_{\gamma \in \Gamma}$ is the upper BZ data of an affine MV polytope so $P$ is in fact an MV polytope. By the definition of $M^1$ and $M^{\overline{1}}$,  $\mu^0 = \mu_w$ thus $P$ is a lower affine MV polytope of highest vertex $w$. 
\end{proof}

Using the BZ data, we can now attempt to answer the following question:

\begin{question}\label{question:L}
Can we find a natural variety $X$ which is a positive space with a potential $\tau$ such that $X(\Z^\trop)_\geq^\tau$ will correspond to lower affine MV polytopes of highest vertex $w$?
\end{question}

Motivated by the finite case in \cite{MVBruhat}, we expect that $X = L^{w^{-1}}$. In the next section, we will find tropical functions $M_\gamma$ that will take a non-negative tropical point to the BZ data of a lower affine MV polytope. 

\section{Tropical geometry of reduced double Bruhat cells}\label{section:tropicalgeometry}

Following \cite[Section 3.4]{Kac-MoodyGroups}, we can define the subgroups $B_-$ and $N$ of the Kac-Moody group $\widehat{SL_2}$. We define the reduced double Bruhat cell $L^{w^{-1}} := N \cap B_- w^{-1} B_-$. 

We want to define the tropical points of $L^{w^{-1}}$ using the Lusztig coordinates. Let $x_i$ be the 1 parameter subgroup associated to $\alpha_i$. Let $\ui = (i_1, \dots, i_m)$ be such that $w = s_{i_1} \dots s_{i_m}$ is a reduced word, and define the function $x_{\ui} : (C^\times)^m \rightarrow L^{w^{-1}}$ by $x_{\ui}(a_1, \dots, a_m) = x_{i_m}(a_m) \dots x_{i_1}(a_1)$. Recall the following proposition of \cite{DoubleBruhatKacMoodyWilliams}:
\begin{proposition}[{\cite[Proposition 4.3]{DoubleBruhatKacMoodyWilliams}}]
The map $x_{\ui} : (\C^\times)^m \rightarrow L^{w^{-1}}$ is a birational isomorphism. 
\end{proposition}
As we are working with $w$ in the affine Weyl group of $\widehat{sl_2}$, there is only one choice of $\ui$ and hence $x_{\ui}$ forms a positive atlas on $L^{w^{-1}}$ as there are no transition functions. Thus the tropical points $L^{w^{-1}}(\Z^\trop)$ are well defined. 

\begin{corollary}\label{corollary:tropicalpoints}
$L^{w^{-1}}(\Z^\trop) \cong \Z^m$.
\end{corollary}
We can write $\ell \in L^{w^{-1}}(\Z^\trop)$ as $\ell = (P_1, \dots, P_m)$, where $P_1, \dots, P_m$ are indexed according to a reduced word $\ui = (i_1, \dots, i_m)$ for $w$. Note that the indices on these $P_i$ are in the opposite order as the terms $a_i$ in the Lusztig coordinates $x_{\ui}(a_1, \dots, a_m)$.

In the next sections, we define tropical functions $M_\gamma$ and use a potential function to determine the non-negative tropical points of $L^{w^{-1}}$. 

\subsection{Tropical generalized minor functions}

Let $\lambda$ be a dominant weight of $\widehat{SL_2}$ and let $V(\lambda)$ be the highest weight representation of weight $\lambda$, with highest weight vector $v_\lambda$. Let $\langle \cdot, \cdot \rangle$ denote the Shapovalov form \cite{Shapovalovform} on $V(\lambda)$, where $\langle F_i v, u \rangle = \langle v, E_i u \rangle$ for every $i \in I, v,u\in V(\lambda)$. For $\delta, \gamma$ extremal weights of $V(\lambda)$, the \emph{generalized minors} are functions $\Delta_{\delta,\gamma}: \widehat{SL_2} \rightarrow \C$ such that 
\[
\Delta_{\delta,\gamma} (g) = \langle g \cdot v_\gamma, v_\delta \rangle.
\]
We denote $\Delta_{\omega_i, \gamma} = \Delta_{\gamma}$ when $\gamma$ is a chamber weight of level $i$.  

For $\gamma \in \Gamma$, the functions $\Delta_{\gamma}$ are defined with respect to extremal weight vectors in the fundamental representations of $\widehat{sl_2}$. These representations can be viewed as subrepresentations of the Fock space as in \cite{KacBook}. We will use the combinatorial approach of \cite{FockSpace} to label vectors of the Fock space as partitions. 

\begin{definition} (Definition 2.2 of \cite{FockSpace})  A \emph{charged partition} $(\lambda, i)$ is a pair consisting of a partition $\lambda$ and a integer $i$, where $i$ is called the \emph{charge}. 

The \textit{Fermionic Fock space} is the free span over all charged partitions. Define
\[
F^{(m)} = \text{span} \{ \text{charged partitions with charge }m\}.
\] 
\end{definition}

We can write the charged partition $(\lambda, k)$ as an upward-facing charged partition where each box of the partition is labelled by a 1 or a 0 by the following rules:
\begin{itemize}
 \item the box in the first row is labelled by the charge
 \item the boxes in each row is constant
 \item the labelling on the rows alternates.
\end{itemize}
 For example, the charged partition ((3,2,2,1),0) is given by:
\[
\begin{tikzpicture}
\draw (0,0) -- (-.5,.5) -- (0,1) -- (.5,.5) -- (0,0); \node at (0, .5) {0};

\draw (-.5,.5) -- (0,1) -- (-.5,1.5) -- (-1,1) -- (-.5,.5);\node at (-.5, 1) {1};
\draw (.5,.5) -- (0,1) -- (.5,1.5) -- (1,1) -- (.5,.5);\node at (.5, 1) {1};

\draw (-1,1) -- (-.5,1.5) -- (-1,2) -- (-1.5,1.5) -- (-1,1);\node at (-1, 1.5) {0};
\draw (0,1) -- (.5,1.5) -- (0,2) -- (-.5,1.5) -- (0,1);\node at (0, 1.5) {0};
\draw (1,1) -- (.5,1.5) -- (1,2) -- (1.5,1.5) -- (1,1);\node at (1, 1.5) {0};

\draw (0.5,1.5) -- (1,2) -- (.5,2.5) -- (0, 2) -- (0.5, 1.5);\node at (.5, 2) {1};
\draw (1.5,1.5) -- (2,2) -- (1.5,2.5) -- (1, 2) -- (1.5, 1.5);\node at (1.5, 2) {1};
\end{tikzpicture}
\]

Let $\omega_i = (\emptyset, i) \in F^{(i)}$. By \cite[Proposition 3.34]{FockSpace}, the fundamental weight representations $V(\omega_i) \subset F^{(i)}$ are generated by the action of the elements $F_1$ and $F_0$ of $\widehat{sl_2}$ on the highest weight vector of weight $\omega_i$. These actions are given by 
\begin{align*}
F_i (\lambda, k) &= \sum_{\substack{\mu\setminus \lambda \text{ is a } \\ i\text{-coloured box}}} (\mu, k), &
E_i(\lambda, k) = \sum_{\substack{\lambda\setminus \mu \text{ is a } \\ i \text{-coloured box}}} (\mu, k). 
\end{align*}
The chamber weights $\gamma= s_{i_1} \cdots s_{i_k} \omega_{i_k} \in \Gamma$ have extremal weight vectors $v_\gamma$ given by the charged partition $((k, k-1, \dots, 2, 1), i_k)$. Below, we denote extremal weight vectors $v_\gamma$ by their weight $\gamma$.

To calculate the value of $\Delta_\gamma$ in the Lusztig coordinates of $N$, we need to understand how $x \in N$ acts on $\gamma$. Using that $x_i(p) = \exp( p E_i)$ and the series form of the exponential, 
\begin{align*}
x_i(p) \cdot \gamma = \sum_{k=0}^\infty \frac{p^k}{k!} E_i^k \cdot \gamma = \sum_{\substack{\mu \text{ obtained by removing} \\ i\text{-coloured boxes from } \gamma}} \frac{p^{|\gamma \setminus \mu|}}{|\gamma \setminus \mu|!} \cdot \mu .
\end{align*}
Then for an arbitrary $x \in N$, 
\begin{align}\label{equation:induction2}
x \cdot x_i(p) \cdot \gamma = \sum_{\substack{ \mu \text{ obtained by removing } \\ i\text{-coloured boxes from } \gamma}} \frac{p^{|\gamma \setminus \mu|}}{|\gamma \setminus \mu|!}  x \cdot \mu .
\end{align}

Inspired by the definition of $\ui$-trails in \cite{Tensorproductmultiplicities} on weights, we define an $\ui$-trail for a partition $\gamma$ to $\delta$ as follows. 

\begin{definition}
Let $\ui = (i_1, \dots, i_m)$. An $\ui$-trail from $\gamma$ to $\delta$ is a sequence of partitions $\mu_i$ such that $\gamma =\mu_0 \supseteq \mu_1 \supseteq \cdots \supseteq \mu_{m-1} \supseteq \mu_{m} = \delta$ and $\mu_{k-1} \setminus \mu_{k}$ are $i_k$-coloured boxes. Define the length of an $\ui$-trail as the number of $k$ such that $\mu_{k} \neq \mu_{k+1}$. 
\end{definition}

\begin{lemma}\label{lemma:positiveminors}
Let $\gamma$ be a chamber weight of level $i$. In the Lusztig coordinates $x_{i_1}(p_1) \cdots x_{i_m}(p_m) \in N$ for a reduced word $\ui$ of $w^{-1}$, the generalized minor takes on the value
\[
\Delta_\gamma(x_{i_1}(p_1) \cdots x_{i_m}(p_m)) = \sum_{\substack{(i_m, \dots, i_1)\text{-trails}\\ \text{from }\gamma \text{ to } \omega_i}} \prod_{s=1}^m \frac{p_{m+1-s}^{|\mu_{s-1} \setminus \mu_s|}}{|\mu_{s-1} \setminus \mu_s|!}.
\] 
\end{lemma}

\begin{proof}
Suppose that $\gamma$ is a chamber weight of level $i$ so that $\Delta_\gamma(g) = \langle g \cdot \gamma, \omega_i \rangle$. We proceed by induction on $\ell(w)=m$. 

When $m = 1$, $\langle x_{i_1}(p_{i_1}) \cdot \gamma, \omega_i \rangle$ picks out the coefficient of $\omega_i$. Thus $\Delta_\gamma (x_{i_1}(p_{i_1})) = p_{i_1}^{|\gamma \setminus \omega_i|}$ if there exists an $i_1$-trail from $\gamma$ to $\omega_i$, otherwise it is zero. 

Suppose that for $x = x_{i_1}(p_1)\cdots x_{i_k}(p_k)$, 
\[
\Delta_\gamma (x) = \langle x \cdot \gamma, \omega_i \rangle = \sum_{\substack{(i_k, \dots, i_1)\text{-trails}\\ \text{from } \gamma \text{ to } \omega_i}} \prod_{s=1}^k \frac{p_{m+1-s}^{|\mu_{i_{s-1}} \setminus \mu_{i_s}|} }{|\mu_{i_{s-1}} \setminus \mu_{i_s}|!}. 
\]
Consider $x \cdot x_{i_{k+1}}(p_{k+1}) \cdot \gamma$. From (\ref{equation:induction2}), 
\begin{align*}
\langle x \cdot x_{i_{k+1}}(p_{k+1}) \cdot \gamma, \omega_i \rangle & = \sum_{\substack{i_{k+1}\text{-trails from} \\ \gamma \text{ to } \mu}} \frac{p_{m+1 - (k+1)}^{|\gamma\setminus\mu|}}{|\gamma\setminus\mu|!} \langle x \cdot \mu, \omega_i \rangle\\
& = \sum_{\substack{i_{k+1}\text{-trails from} \\ \gamma \text{ to } \mu}} \left(\sum_{\substack{(i_k, \dots, i_1)\text{-trails}\\ \text{from } \mu \text{ to } \omega_i}} \prod_{s=1}^k \frac{p_{m+1-s}^{|\mu_{s-1} \setminus \mu_s|}}{|\mu_{s-1} \setminus \mu_s|!} \right)  \frac{p_{m+1-(k+1)}^{|\gamma\setminus\mu|}}{|\gamma\setminus\mu|!}\\
& = \sum_{\substack{(i_{k+1}, i_k, \dots, i_1)\text{-trails} \\ \text{from }\gamma \text{ to } \omega_i}}\prod_{s=1}^{k+1} \frac{p_{m+1-s}^{|\mu_{s-1} \setminus \mu_s|}}{|\mu_{s-1} \setminus \mu_s|!}. \qedhere
\end{align*}
\end{proof}

As the generalized minors can be written as subtraction-free expressions in the Lusztig coordinates, they are positive functions and thus can be tropicalized. Define the function $M_\gamma = \Delta^\trop_\gamma$ as the tropical function on $L^{w^{-1}}(\Z^\trop)$. We will show that these tropical functions satisfy the same relations as in Lemma \ref{lemma:BZdata}, and thus the set $(M_\gamma(\ell))_{\gamma \in \Gamma}$ will be the BZ datum of an affine MV polytope of highest vertex $w$ for any non-negative tropical point $\ell \in L^{w^{-1}}$. 

In the reduced word $\ui$ for $w$, these generalized minor functions take on the value
\[
\Delta_\gamma(x_{\ui}(p_1, \dots, p_m)) =  \sum_{\substack{\ui\text{-trails}\\ \text{from }\gamma \text{ to } \omega_i}} \prod_{s=1}^m \frac{p_{s}^{|\mu_{s-1} \setminus \mu_s|}}{|\mu_{s-1} \setminus \mu_s|!}.
\]
By tropicalizing this relation, we give explicit formulae for these tropicalized functions. 
\begin{lemma}\label{lemma:Mequations}
Let $\ui = (i_1, \dots, i_m)$ be a reduced word for $w$ and let $(P_1, \dots, P_m) \in L^{w^{-1}}(\Z^\trop)$. For any $k \in \N$,
\begin{align*}
M_k (P_1, \dots, P_m) & =   \min_{\substack{ 1\leq j_1<j_2<\cdots<j_{k-1}\leq m \\ i_{j_\ell} =0,\text{ if } j_\ell=0 \mod 2\\ i_{j_\ell} =1,\text{ if } j_\ell=1 \mod 2}} \left\{ \sum_{s=1}^{k-1} (k-s) P_{j_s} \right\}, \\
M_{\bar{k}} (P_1, \dots, P_m) & = \min_{\substack{ 1\leq j_1<j_2<\cdots<j_{k-1} \leq m \\ i_{j_\ell} =0,\text{ if } j_\ell=1 \mod 2\\ i_{j_\ell} =1,\text{ if } j_\ell=0 \mod 2}} \left\{ \sum_{s=1}^{k-1} (k-s) P_{j_s} \right\}.
\end{align*}
\end{lemma}
Note that when $\gamma_k, \ogamma_k \not\in \Gamma^w$, then the set $\{(j_1, \dots, j_k): 1\leq j_1  < j_2 < \cdots <j_k \leq m\}$ is empty. As the minimum is taken over an empty set,  the corresponding function is equal to $\infty$. 

\begin{proof} The tropicalization of $\Delta_\gamma$ is
\begin{align}\label{eqn:directtrop}
M_\gamma (P_1, \dots, P_m) & = \min_{\substack{\ui\text{-trails} \\ \text{from } \gamma \text{ to } \omega_i}} \left\lbrace \sum_{s=1}^m |\mu_{s-1} \setminus \mu_s| P_s \right\rbrace .
\end{align}
All we need to show is the minimum is only dependent on length $k$ $\ui$-trails from $\gamma$ to $\omega_i$. 

Consider an $\ui$-trail $\gamma = \mu_0 \supseteq \mu_{1} \supseteq \cdots \supseteq \mu_{m-1}\supseteq \mu_m= \omega_i$ and the tropical value $\sum_{s=1}^m |\mu_{s-1} \setminus \mu_s| P_s$. We label the partition $\gamma$ by labelling the boxes in $\mu_{s-1}\setminus \mu_{s}$ by $P_s$. Denote the labelling on box $i$ in row $j$ by $L(i,j)$. Define a path $\sigma$ in $\gamma$ as a collection of boxes from each row of the partition such that all the boxes are connected in $\gamma$, i.e. $\sigma = (\sigma(1), \cdots, \sigma(k))$ such that $1\leq \sigma(i)\leq i$ is the choice of box in row $i$ and $\sigma(i+1)$ is either $\sigma(i)$ or $\sigma(i)+1$. Then $L(\sigma(i), i)$ is the labelling in box $\sigma(i)$ in row $i$. 

Let $\chi$ be a path in $\gamma$ such that 
\[
L(\chi(1),1) + \cdots + L(\chi(k),k) = \min_{\sigma \text{ path in } \gamma} \{ L(\sigma(1),1) + \cdots + L(\sigma(k),k)\},
\]
the path with the minimal sum of its labelling. We will show that $\sum_{j=1}^m j L(\chi(j),j) \leq \sum_{s=1}^m |\mu_{s-1}\setminus \mu_s| P_s$. 

First, we define paths $\sigma_i$ in $\gamma$. Let $\sigma_i$ be the path such that $\sigma_i(j) = \chi(j)$ for $1 \leq j \leq i$. For $i+1 \leq j \leq k$, if $\chi(i+1) = \chi(i)$, then we set $\sigma_i(j) = \sigma_i(i)+(j-i)$. Otherwise, $\chi(i+1) = \chi(i)+1$ and we set $\sigma_i(j) = \sigma_i(i)$.

From this definition, $\sum_{j=1}^i L(\chi(j), j) = \sum_{j=1}^i L(\sigma_i(j), j)$. Since $\chi$ is a path with the minimal labelling, this implies that 
\[
\sum_{j=i+1}^k L(\chi(j), j) \leq \sum_{j=i+1}^k L(\sigma_i(j),j)
\]
for every $1 \leq i < k$. It immediately follows that
\[
\sum_{j=1}^k (j-1)L(\chi(j), j) = \sum_{i=1}^{k-1} \sum_{j=i+1}^k L(\chi(j),j) \leq \sum_{i=1}^{k-1} \sum_{j=i+1}^k L(\sigma_i(j),j).
\]
Since $L(\chi(i),i) = L(\sigma_i(i),i)$ for $1 \leq i \leq k$, then  $\sum_{j=1}^{k} j L(\chi(j),j) \leq \sum_{i=1}^k \sum_{j=i}^k L(\sigma_i(j),j)$. All that is left is to show that $\sum_{s=1}^m |\mu_{s-1} \setminus \mu_s| P_s = \sum_{s=1}^k \sum_{j=s}^k L(\sigma_s(j), j)$.

\begin{claim}
 Let $S_i = \{ (\sigma_i(j),j) : i \leq j \leq k\}$. The sets $\{S_i\}_{i=1}^k$ partition $\gamma$. 
 \end{claim}

\begin{claimproof}
First, notice that the number of boxes in the set $\{ (\sigma_i(j),j): 1 \leq i\leq k, i\leq j\leq k\}$ is
\[
\sum_{i=1}^k \sum_{j=i}^k 1 = \sum_{i=1}^k (k+1-i) = \sum_{i=1}^k i 
\]
which is exactly the number of boxes in $\gamma$. 

Second, we show that no path $\sigma_i$ contains a box $\sigma_i(j)$ in $\chi$ for $i+1\leq j\leq k$.  If $\sigma_i(j) = \sigma_i(i)$, then $\chi(i+1) = \chi(i)+1$ and so $\chi(j) \geq \chi(i)+1 > \sigma_i(j)$. If $\sigma_i(j) = \sigma_i(i) + (j-i)$, then $\chi(i+1) = \chi(i)$ so $\chi(j) < \chi(i) + (j-i) = \sigma_i(j)$. Thus $\sigma_i$ doesn't intersect with $\chi$ above row $i$ so $S_i \cap \{ (\chi(j),j): 1 \leq j \leq k\} = (\chi(i),i)$. 

Third, we show for $i<j$, no two paths $\sigma_i$, $\sigma_j$ intersect above row $j$. Let $j \leq \ell \leq k$ and consider $\sigma_i(\ell)$. If $\sigma_i(\ell) = \sigma_i(i)$, then $\chi(i+1) = \chi(i)+1$ so $\chi(s) > \chi(i)$ for every $s >i$. Then
\[
\sigma_j(\ell) \geq \sigma_j(j) = \chi(j) > \chi(i) = \sigma_i(i) = \sigma_i(\ell)
\]
If $\sigma_i(\ell) = \sigma_i(i) + (\ell-i)$, then $\chi(i+1) = \chi(i)$ so $\chi(s) < \chi(i) + (s -i)$ for every $s>i$. Then
\[
\sigma_j(\ell) \leq \sigma_j(j) + (\ell -j)  = \chi(j) + (\ell-j) < \chi(i) + (j-i) + (\ell -j) = \sigma_i(i) + (\ell-i) = \sigma_i(\ell)
\]
Thus $\sigma_i(\ell) \neq \sigma_j(\ell)$ as desired.

The second and third points tell us that $S_i \cap S_j =\emptyset$ for every $i \neq j$ and hence the sets $S_i$ give a partition of $\gamma$.
\end{claimproof}

As the sets $S_i$ partition $\gamma$, then $\sum_{s=1}^{m} |\mu_{s-1} \setminus \mu_s | P_s = \sum_{i=1}^k \sum_{j=1}^k L(\sigma_s(j),j)$ as both are equal to the sum of all the labels on $\gamma$.

Consider a new labeling of $\gamma$ by $L'(i,j) = L(\chi(j),j)$, i.e. we label each box in row $j$ by $L(\chi(j),j)$. The corresponding $\ui$-trail of length $k$ for $\gamma$ will give the tropical value $\sum_{j=1}^m j L(\chi(j),j)$. As we showed above, this is smaller than the value $\sum_{s=1}^m |\mu_{s-1}\setminus \mu_s| P_s$ and hence for any $\ui$-trail, we can find a length $k$ $\ui$-trail which will give a smaller value. Thus the minimum in the tropical functions $M_\gamma$ depend only on length $k$ $\ui$-trails. 
\end{proof}

\begin{example}
Consider the reduced word $\ui= (1,0,1,0,1,0,1)$ and the following $\ui$-trail:
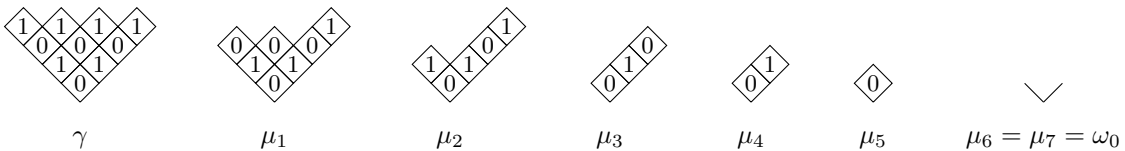
\begin{align*}
\begin{tikzpicture}
\draw (0,0) -- (-.25,.25) -- (0,.5) -- (.25,.25) -- (0,0); \node at (0, .25) {\small{0}};
\draw (-.25,.25) -- (0,.5) -- (-.25,.75) -- (-.5,.5) -- (-.25,.25);\node at (-.25, .5) {\small{1}};
\draw (.25,.25) -- (0,.5) -- (.25,.75) -- (.5,.5) -- (.25,.25);\node at (.25, .5) {\small{1}};
\draw (-.5,.5) -- (-.25,.75) -- (-.5,1) -- (-.75,.75) -- (-.5,.5);\node at (-.5, .75) {\small{0}};
\draw (0,.5) -- (-.25,.75) -- (0,1) -- (.25,.75) -- (0,.5); \node at (0, .75) {\small{0}};
\draw (.5,.5) -- (.25,.75) -- (.5,1) -- (.75,.75) -- (.5,.5);\node at (.5, .75) {\small{0}}; 
\draw (-.75,.75) -- (-0.5,1) -- (-.75,1.25) -- (-1,1) -- (-.75,.75);\node at (-.75, 1) {\small{1}};
\draw (-.25,.75) -- (0,1) -- (-.25,1.25) -- (-.5,1) -- (-.25,.75);\node at (-.25, 1) {\small{1}};
\draw (.25,.75) -- (0,1) -- (.25,1.25) -- (.5,1) -- (.25,.75);\node at (.25, 1) {\small{1}};
\draw (.75,.75) -- (0.5,1) -- (.75,1.25) -- (1,1) -- (.75,.75);\node at (.75, 1) {\small{1}};
\node at (0,-.5) {$\gamma$};
\end{tikzpicture}
&&
\begin{tikzpicture}
\draw (0,0) -- (-.25,.25) -- (0,.5) -- (.25,.25) -- (0,0); \node at (0, .25) {\small{0}};
\draw (-.25,.25) -- (0,.5) -- (-.25,.75) -- (-.5,.5) -- (-.25,.25);\node at (-.25, .5) {\small{1}};
\draw (.25,.25) -- (0,.5) -- (.25,.75) -- (.5,.5) -- (.25,.25);\node at (.25, .5) {\small{1}};
\draw (-.5,.5) -- (-.25,.75) -- (-.5,1) -- (-.75,.75) -- (-.5,.5);\node at (-.5, .75) {\small{0}};
\draw (0,.5) -- (-.25,.75) -- (0,1) -- (.25,.75) -- (0,.5); \node at (0, .75) {\small{0}};
\draw (.5,.5) -- (.25,.75) -- (.5,1) -- (.75,.75) -- (.5,.5);\node at (.5, .75) {\small{0}}; 
\draw (.75,.75) -- (0.5,1) -- (.75,1.25) -- (1,1) -- (.75,.75);\node at (.75, 1) {\small{1}};
\node at (0,-.5)  {$\mu_1$};
\end{tikzpicture}
&&
\begin{tikzpicture}
\draw (0,0) -- (-.25,.25) -- (0,.5) -- (.25,.25) -- (0,0); \node at (0, .25) {\small{0}};
\draw (-.25,.25) -- (0,.5) -- (-.25,.75) -- (-.5,.5) -- (-.25,.25);\node at (-.25, .5) {\small{1}};
\draw (.25,.25) -- (0,.5) -- (.25,.75) -- (.5,.5) -- (.25,.25);\node at (.25, .5) {\small{1}};
\draw (.5,.5) -- (.25,.75) -- (.5,1) -- (.75,.75) -- (.5,.5);\node at (.5, .75) {\small{0}}; 
\draw (.75,.75) -- (0.5,1) -- (.75,1.25) -- (1,1) -- (.75,.75);\node at (.75, 1) {\small{1}};
\node at (0,-.5) {$\mu_2$};
\end{tikzpicture}
&&
\begin{tikzpicture}
\draw (0,0) -- (-.25,.25) -- (0,.5) -- (.25,.25) -- (0,0); \node at (0, .25) {\small{0}};
\draw (.25,.25) -- (0,.5) -- (.25,.75) -- (.5,.5) -- (.25,.25);\node at (.25, .5) {\small{1}};
\draw (.5,.5) -- (.25,.75) -- (.5,1) -- (.75,.75) -- (.5,.5);\node at (.5, .75) {\small{0}}; 
\node at (0,-.5) {$\mu_3$};
\end{tikzpicture}
&&
\begin{tikzpicture}
\draw (0,0) -- (-.25,.25) -- (0,.5) -- (.25,.25) -- (0,0); \node at (0, .25) {\small{0}};
\draw (.25,.25) -- (0,.5) -- (.25,.75) -- (.5,.5) -- (.25,.25);\node at (.25, .5) {\small{1}};
\node at (0, -.5) {$\mu_4$};
\end{tikzpicture}
&&
\begin{tikzpicture}
\draw (0,0) -- (-.25,.25) -- (0,.5) -- (.25,.25) -- (0,0); \node at (0, .25) {\small{0}};
\node at (0, -0.5) {$\mu_5$};
\end{tikzpicture}
&&
\begin{tikzpicture}
\draw (0,0) -- (-.25,.25);
\draw (.25,.25) -- (0,0);
\node at (0, -.5) {$\mu_6=\mu_7 = \omega_0$};
\end{tikzpicture}
\end{align*}

We label the partition $\gamma$ as 
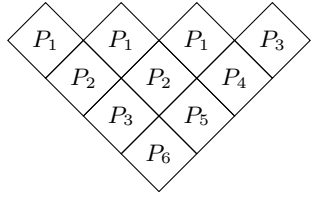
\[
\begin{tikzpicture}
\draw (0,0) -- (-.5,.5) -- (0,1) -- (.5,.5) -- (0,0); \node at (0, .5) {\small{$P_6$}};
\draw (-.5,.5) -- (0,1) -- (-.5,1.5) -- (-1,1) -- (-.5,.5);\node at (-.5, 1) {\small{$P_3$}};
\draw (.5,.5) -- (0,1) -- (.5,1.5) -- (1,1) -- (.5,.5);\node at (.5,1) {\small{$P_5$}};
\draw (-1,1) -- (-.5,1.5) -- (-1,2) -- (-1.5,1.5) -- (-1,1);\node at (-1,1.5) {\small{$P_2$}};
\draw (0,1) -- (-.5,1.5) -- (0,2) -- (.5,1.5) -- (0,1); \node at (0,1.5) {\small{$P_2$}};
\draw (1,1) -- (.5,1.5) -- (1,2) -- (1.5,1.5) -- (1,1);\node at (1,1.5) {\small{$P_4$}}; 
\draw (-1.5,1.5) -- (-1,2) -- (-1.5,2.5) -- (-2,2) -- (-1.5,1.5);\node at (-1.5, 2) {\small{$P_1$}};
\draw (-.5,1.5) -- (0,2) -- (-.5,2.5) -- (-1,2) -- (-.5,1.5);\node at (-.5, 2) {\small{$P_1$}};
\draw (.5,1.5) -- (0,2) -- (.5,2.5) -- (1,2) -- (.5,1.5);\node at (.5, 2) {\small{$P_1$}};
\draw (1.5,1.5) -- (1,2) -- (1.5,2.5) -- (2,2) -- (1.5,1.5);\node at (1.5, 2) {\small{$P_3$}};
\end{tikzpicture}
\]
If $\chi = \begin{tikzpicture} \draw (0,0) -- (-.25,.25) -- (0,.5) -- (.25,.25) -- (0,0);  \node at (0, .25) {\small{0}};
\draw (.25,.25) -- (0,.5) -- (.25,.75) -- (.5,.5) -- (.25,.25);\node at (.25, .5) {\small{1}};
\draw (.5,.5) -- (.25,.75) -- (.5,1) -- (.75,.75) -- (.5,.5);\node at (.5, .75) {\small{0}}; 
\draw (.25,.75) -- (0,1) -- (.25,1.25) -- (.5,1) -- (.25,.75);\node at (.25, 1) {\small{1}};
\end{tikzpicture}$, then the paths $\sigma_i$ in the proof of Lemma  \ref{lemma:Mequations} are:
\begin{align*}
\begin{tikzpicture}
\draw (0,0) -- (-.25,.25) -- (0,.5) -- (.25,.25) -- (0,0); \node at (0, .25) {\small{0}};
\draw (-.25,.25) -- (0,.5) -- (-.25,.75) -- (-.5,.5) -- (-.25,.25);\node at (-.25, .5) {\small{1}};
\draw (-.5,.5) -- (-.25,.75) -- (-.5,1) -- (-.75,.75) -- (-.5,.5);\node at (-.5, .75) {\small{0}};
\draw (-.75,.75) -- (-0.5,1) -- (-.75,1.25) -- (-1,1) -- (-.75,.75);\node at (-.75, 1) {\small{1}};
\node at (0,-.5) {$\sigma_1$};
\end{tikzpicture} & &
\begin{tikzpicture}
\draw (0,0) -- (-.25,.25) -- (0,.5) -- (.25,.25) -- (0,0); \node at (0, .25) {\small{0}};
\draw (.25,.25) -- (0,.5) -- (.25,.75) -- (.5,.5) -- (.25,.25);\node at (.25, .5) {\small{1}};
\draw (0,.5) -- (-.25,.75) -- (0,1) -- (.25,.75) -- (0,.5); \node at (0, .75) {\small{0}};
\draw (-.25,.75) -- (0,1) -- (-.25,1.25) -- (-.5,1) -- (-.25,.75);\node at (-.25, 1) {\small{1}};
\node at (0,-.5) {$\sigma_2$};
\end{tikzpicture} & &
\begin{tikzpicture}
\draw (0,0) -- (-.25,.25) -- (0,.5) -- (.25,.25) -- (0,0); \node at (0, .25) {\small{0}};
\draw (.25,.25) -- (0,.5) -- (.25,.75) -- (.5,.5) -- (.25,.25);\node at (.25, .5) {\small{1}};
\draw (.5,.5) -- (.25,.75) -- (.5,1) -- (.75,.75) -- (.5,.5);\node at (.5, .75) {\small{0}}; 
\draw (.75,.75) -- (0.5,1) -- (.75,1.25) -- (1,1) -- (.75,.75);\node at (.75, 1) {\small{1}};
\node at (0,-.5) {$\sigma_3$};
\end{tikzpicture} &&
\begin{tikzpicture}
\draw (0,0) -- (-.25,.25) -- (0,.5) -- (.25,.25) -- (0,0); \node at (0, .25) {\small{0}};
\draw (.25,.25) -- (0,.5) -- (.25,.75) -- (.5,.5) -- (.25,.25);\node at (.25, .5) {\small{1}};
\draw (.5,.5) -- (.25,.75) -- (.5,1) -- (.75,.75) -- (.5,.5);\node at (.5, .75) {\small{0}}; 
\draw (.25,.75) -- (0,1) -- (.25,1.25) -- (.5,1) -- (.25,.75);\node at (.25, 1) {\small{1}};
\node at (0,-.5) {$\sigma_4 = \chi$};
\end{tikzpicture}
\end{align*}
Then the labelling of $\gamma$ by
\[
\begin{tikzpicture}
\draw (0,0) -- (-.5,.5) -- (0,1) -- (.5,.5) -- (0,0); \node at (0, .5) {\small{$P_6$}};
\draw (-.5,.5) -- (0,1) -- (-.5,1.5) -- (-1,1) -- (-.5,.5);\node at (-.5, 1) {\small{$P_5$}};
\draw (.5,.5) -- (0,1) -- (.5,1.5) -- (1,1) -- (.5,.5);\node at (.5,1) {\small{$P_5$}};
\draw (-1,1) -- (-.5,1.5) -- (-1,2) -- (-1.5,1.5) -- (-1,1);\node at (-1,1.5) {\small{$P_4$}};
\draw (0,1) -- (-.5,1.5) -- (0,2) -- (.5,1.5) -- (0,1); \node at (0,1.5) {\small{$P_4$}};
\draw (1,1) -- (.5,1.5) -- (1,2) -- (1.5,1.5) -- (1,1);\node at (1,1.5) {\small{$P_4$}}; 
\draw (-1.5,1.5) -- (-1,2) -- (-1.5,2.5) -- (-2,2) -- (-1.5,1.5);\node at (-1.5, 2) {\small{$P_1$}};
\draw (-.5,1.5) -- (0,2) -- (-.5,2.5) -- (-1,2) -- (-.5,1.5);\node at (-.5, 2) {\small{$P_1$}};
\draw (.5,1.5) -- (0,2) -- (.5,2.5) -- (1,2) -- (.5,1.5);\node at (.5, 2) {\small{$P_1$}};
\draw (1.5,1.5) -- (1,2) -- (1.5,2.5) -- (2,2) -- (1.5,1.5);\node at (1.5, 2) {\small{$P_1$}};
\end{tikzpicture}
\]
with the $\ui$-trail 
\begin{align*}
\begin{tikzpicture}
\draw (0,0) -- (-.25,.25) -- (0,.5) -- (.25,.25) -- (0,0); \node at (0, .25) {\small{0}};
\draw (-.25,.25) -- (0,.5) -- (-.25,.75) -- (-.5,.5) -- (-.25,.25);\node at (-.25, .5) {\small{1}};
\draw (.25,.25) -- (0,.5) -- (.25,.75) -- (.5,.5) -- (.25,.25);\node at (.25, .5) {\small{1}};
\draw (-.5,.5) -- (-.25,.75) -- (-.5,1) -- (-.75,.75) -- (-.5,.5);\node at (-.5, .75) {\small{0}};
\draw (0,.5) -- (-.25,.75) -- (0,1) -- (.25,.75) -- (0,.5); \node at (0, .75) {\small{0}};
\draw (.5,.5) -- (.25,.75) -- (.5,1) -- (.75,.75) -- (.5,.5);\node at (.5, .75) {\small{0}}; 
\draw (-.75,.75) -- (-0.5,1) -- (-.75,1.25) -- (-1,1) -- (-.75,.75);\node at (-.75, 1) {\small{1}};
\draw (-.25,.75) -- (0,1) -- (-.25,1.25) -- (-.5,1) -- (-.25,.75);\node at (-.25, 1) {\small{1}};
\draw (.25,.75) -- (0,1) -- (.25,1.25) -- (.5,1) -- (.25,.75);\node at (.25, 1) {\small{1}};
\draw (.75,.75) -- (0.5,1) -- (.75,1.25) -- (1,1) -- (.75,.75);\node at (.75, 1) {\small{1}};
\node at (0,-.5) {$\gamma$};
\end{tikzpicture}
&&
\begin{tikzpicture}
\draw (0,0) -- (-.25,.25) -- (0,.5) -- (.25,.25) -- (0,0); \node at (0, .25) {\small{0}};
\draw (-.25,.25) -- (0,.5) -- (-.25,.75) -- (-.5,.5) -- (-.25,.25);\node at (-.25, .5) {\small{1}};
\draw (.25,.25) -- (0,.5) -- (.25,.75) -- (.5,.5) -- (.25,.25);\node at (.25, .5) {\small{1}};
\draw (-.5,.5) -- (-.25,.75) -- (-.5,1) -- (-.75,.75) -- (-.5,.5);\node at (-.5, .75) {\small{0}};
\draw (0,.5) -- (-.25,.75) -- (0,1) -- (.25,.75) -- (0,.5); \node at (0, .75) {\small{0}};
\draw (.5,.5) -- (.25,.75) -- (.5,1) -- (.75,.75) -- (.5,.5);\node at (.5, .75) {\small{0}}; 
\node at (0,-.5) {$\mu_1 = \mu_2 = \mu_3$};
\end{tikzpicture}
&&
\begin{tikzpicture}
\draw (0,0) -- (-.25,.25) -- (0,.5) -- (.25,.25) -- (0,0); \node at (0, .25) {\small{0}};
\draw (-.25,.25) -- (0,.5) -- (-.25,.75) -- (-.5,.5) -- (-.25,.25);\node at (-.25, .5) {\small{1}};
\draw (.25,.25) -- (0,.5) -- (.25,.75) -- (.5,.5) -- (.25,.25);\node at (.25, .5) {\small{1}};
\node at (0,-.5){$\mu_4$};
\end{tikzpicture}
&&
\begin{tikzpicture}
\draw (0,0) -- (-.25,.25) -- (0,.5) -- (.25,.25) -- (0,0); \node at (0, .25) {\small{0}};
\node at (0,-.5){$\mu_5$};
\end{tikzpicture}
&& 
\begin{tikzpicture}
\draw (0,0) -- (-.25,.25);
\draw (.25,.25) -- (0,0);
\node at (0, -.5) {$\mu_6 = \mu_7 = \omega_0$};
\end{tikzpicture}
\end{align*}
will give a smaller tropical value. 
\end{example}

As a result of this lemma, the values of the $M_\gamma$ functions depends only on subsequences of $(P_1, \cdots, P_m)$ with indices $1 \leq j_1 < j_2< \cdots < j_{k} \leq m$ which alternate between even and odd numbers, beginning with the correct parity; when $\gamma = \ogamma_k$, $j_1$ is even and when $\gamma = \gamma_k$, $j_1$ is odd. As shorthand, we will call such a sequence of $j$'s a \emph{valid sequence for} $M_\gamma$. If $(j_1, \dots, j_{k-1})$ is a sequence of indices such that $M_\gamma = \sum_{s=1}^{k-1} (k-s)P_{j_s}$, then we will call it a \emph{sequence which minimizes} $M_\gamma$. 

Using these equations for $M_\gamma$, we can show that the lower diagonals \ref{condition:diagonal} in Definition \ref{definition:affinepolytope} are satisfied.

\begin{lemma}\label{lemma:LWdiagonals}
On $L^{w^{-1}}(\Z^\trop)$, the functions $M_k, M_{\overline{k}}$ satisfy
\begin{align}\label{eqn:LWdiagonals}
  k M_k + k M_{\overline{k}} = \min \lbrace k M_{k-1} + M_k + (k-1) M_{\overline{k+1}}, k M_{\overline{k-1}} + M_{\overline{k}} + (k-1) M_{k+1} \}
\end{align}
for all $k \geq 2$
\end{lemma}

\begin{proof}
We begin with the edge cases: when there does not exist a valid sequence for one of the terms. When $k \leq m$, there always exists at least one valid sequence for both $M_k$ and $M_{\overline{k}}$. When $k=m+1$, then there is only one valid sequence of length $m$, which is $(1, \dots, m)$ so that $M_{m+1}$ is non-infinite but $M_{\overline{m+1}} = \infty$. When $k \geq m+2$, then both $M_k$ and $M_{\overline{k}}$ are infinite. Thus for any $k \geq m+1$, both sides of (\ref{eqn:LWdiagonals}) are infinite. The only $k$ left to check is $k=m$. In this case, 
\begin{align*}
m( M_{\overline{m}} + M_{m} ) & = m \left( \sum_{i=2}^{m} (m+1-i) P_i + \sum_{i=1}^{m-1} (m-i) P_i \right) \\
& = m\left( (m-1) P_1 + \sum_{i=2}^{m-1} (2m+1-2i) P_i + P_m \right)
\end{align*}
while
\begin{align*}
mM_{m-1} + M_m + (m-1) M_{\overline{m+1}} &= \infty\\
m M_{\overline{m-1}} + M_{\overline{m}} + (m-1) M_{m+1} & = m \sum_{i=2}^{m-1} (m-i) P_i + \sum_{i=2}^{m} (m+1-i) P_i + (m-1)\sum_{i=1}^m (m+1-i) P_i \\
& = m(m-1)P_1  \sum_{i=2}^{m-1} m(2m-2i+1) P_i  +  m P_m
\end{align*}
so that (\ref{eqn:LWdiagonals}) indeed holds for the edge cases. 

For the rest of the proof, we assume that every function has a valid sequence. We begin by proving that $kM_k + k M_{\overline{k}}$ is a lower bound for both terms. We only show the bound for one term; by symmetry, the bound on the other term is an identical argument.

The idea of the proof is as follows. Using the values in $kM_{k-1} + M_k + (k-1) M_{\overline{k+1}}$, we build $k$ valid fillings of $\gamma_k$ and $\gamma_{\overline{k}}$. By using the direct tropicalization (\ref{eqn:directtrop}), we show that these fillings will be bounded below by the values of $M_k$ and $M_{\overline{k}}$. 

Let $(a_1, \dots, a_k)$ be a sequence minimizing $M_{\overline{k+1}}$ and let $(b_1, \dots, b_{k-2})$ be a sequence minimizing $M_{k-1}$. Note that the pairs $a_i$ and $b_i$ have opposite parity while $b_i$ and $a_{i+1}$ have the same parity, for $1 \leq i \leq k-2$. 

Suppose there exists an $1 \leq \ell \leq k-2$ such that $a_\ell < b_\ell$ but $b_i < a_i$ for all $1 \leq i < \ell$. If no such $\ell$ exists, set $\ell = k-1$. Consider the following rearrangement of terms.
\begin{align*}
k M_{k-1} + (k-1) M_{\overline{k+1}} & = k \sum_{i=1}^{k-2} (k-1-i) P_{b_i} + (k-1) \sum_{i=1}^{k} (k+1-i) P_{a_i} \\
& = k  \sum_{i=1}^{\ell-1} (k-1-i) P_{b_i} + k  \sum_{i=\ell+1}^{k-1} (k-i) P_{b_{i-1}} + (k-1) \sum_{i=1}^{\ell}(k+1-i) P_{a_i} \\
& \quad + (k-1)\sum_{i=\ell}^{k-1} (k-i) P_{a_{i+1}} \\
& = k \left( \sum_{i=1}^{\ell}(k-i)P_{a_i} + \sum_{i=\ell+1}^{k-1} (k-i) P_{b_{i-1}} \right) +  \sum_{i=1}^{\ell-1} \left( k(k-i-1) P_{b_i} + i P_{a_{i+1}} \right)\\
& \quad + (k-1)\sum_{i=\ell}^{k-1} (k-i) P_{a_{i+1}} 
\end{align*}
As $a_1 < a_2 < \dots < a_\ell < b_{\ell} < \dots < b_{k-2}$, then $(a_1, \dots, a_\ell, b_\ell, \dots, b_{k-2})$ is a valid sequence for $M_{\overline{k}}$ so that we obtain the following lower bound:
\[
k M_{k-1} + M_k + (k-1) M_{\overline{k+1}} \geq k M_{\overline{k}} + M_k +  \sum_{i=1}^{\ell-1} \left( k(k-i-1) P_{b_i} + i P_{a_{i+1}} \right) + (k-1)\sum_{i=\ell}^{k-1} (k-i) P_{a_{i+1}} .
\]
By the choice of $\ell$, we have $b_{\ell-1} < a_{\ell-1} < a_{\ell} < a_{\ell+1}$, so that $b_1 < b_2< \dots < b_{\ell-1} < a_{\ell+1} < \dots < a_k$ is a valid sequence for $M_k$. 
\begin{align*}
k M_{k-1} + M_k + (k-1) M_{\overline{k+1}} & \geq k M_{\overline{k}} + M_k + (k-1-\ell) \left( \sum_{i=1}^{\ell-1} (k-i) P_{b_i} + \sum_{i=\ell}^{k-1}(k-i) P_{a_{i+1}} \right) \\
& \quad + \sum_{i=1}^{\ell-1}\left( ((k - i)\ell - i) P_{b_i} + i P_{a_{i+1}} \right)+ \ell \sum_{i=\ell}^{k-1} (k-i)P_{a_{i+1}} \\
& \geq k M_{\overline{k}} + (k-\ell) M_k + \sum_{i=1}^{\ell-1}\left( ((k - i)\ell - i) P_{b_i} + i P_{a_{i+1}} \right)+ \ell \sum_{i=\ell}^{k-1} (k-i)P_{a_{i+1}} 
\stepcounter{equation}\tag{\theequation}\label{eqn:prooflower}
\end{align*}
The leftover terms can be arranged into $\ell$ valid fillings of $\gamma_k$, where valid means that the filling is the result of an $\ui\text{-trail}$ from $\gamma_k$ to the appropriate $\omega_i$. 

For the $p$-th filling, let $L(i,j)$ be the index of the labelling of box $i$ in row $j$; for example, $L(i,j) = a$ means we label box $i$ in row $j$ with $P_a$. Recall that we count rows from the bottom so that row $j$ has exactly $j$ boxes. Below, we will call a filling of row $j$ valid if
\begin{enumerate}
 \item $L(i,j)<L(i,j-1)$ for $1 \leq i \leq j-1$ and $L(i, j)<L(i-1, j-1)$ for $2 \leq i \leq j$,
 \item $(-1)^{L(i,j)} = (-1)^{b_{k-j}}$ for $2 \leq i \leq j \leq k-1$ and $(-1)^{L(1,1)} = -(-1)^{b_{k-2}}$. 
\end{enumerate}
The first condition ensures that the labelling of the boxes increases along any path in $\gamma_k$, while the second condition ensures that each label has the correct parity with respect to the charge of $\gamma_k$. If each row of a filling is valid, it follows that the filling is valid.  

In row $1 \leq j \leq k-\ell$, let $L(i,j) = a_{k+1-j}$ for $1 \leq i \leq j$. These row fillings are clearly valid as $(-1)^{a_{k+1-j}} = (-1)^{b_{k-j}}$. 

In row $k-\ell+1 \leq j \leq k-\ell+p-1$, let $L(i,j) = b_{k-j}$ for $i \leq j-1$ and $L(j,j) = a_{k-j+1}$; clearly, these values have the correct parity. For $j= k-\ell+1$, by the choice of $\ell$, $b_{\ell-1} < a_{\ell+1}$ so that this row is valid. For $j > k -\ell+1$, the only potential problem is $L(j-1, j) = b_{k-j}$ and $L(j-1, j-1) = a_{k-j+2}$. But $k-j < k -(k-\ell+1) = \ell-1$, so that $b_{k-j} < a_{k-j+2}$ by the choice of $\ell$, and this filling is valid. 

Finally, in rows $k-\ell+p \leq j \leq k-1$, set $L(i,j) = b_{k-j}$ for $ 1 \leq j$. Again, the only problem is in row $k - \ell+p$ at boxes $j-1$ and $j$. But $L(j,j) = L(j-1, j) = b_{\ell-p} <a_{\ell-p+2} = L(j-1,j-1)$. Thus each of the $p$ fillings are valid. 

For a visualization of the $p$-th filling, see Figure \ref{figure:gammakfilling}. The blue lines indicate boxes labelled with $a_i$'s. Note that the bottom of the partition is labelled with $a_i$'s and the top is labelled with $b_i$'s, with the exception of one arm whose length is dependent on $p$.

\begin{figure}
\[
 \begin{tikzpicture}[scale =1.2]

\draw (-5.75,5.75) -- (-5.25,6.25) -- (-5.75,6.75) -- (-6.25,6.25) --  (-5.75,5.75); \node at (-5.75,6.25) {\tiny{$b_1$}};
\node at (-4.5, 6.25) {{$\cdots$}};
\node at (2.75, 6.25) {{$\cdots$}};
\draw (4.25,6.25) -- (3.75,5.75)  -- (3.25,6.25) -- (3.75,6.75) --  (4.25,6.25); \node at (3.75,6.25) {\tiny{$b_1$}};
\draw (5.25,6.25) -- (4.75,5.75)  -- (4.25,6.25) -- (4.75,6.75) --  (5.25,6.25); \node at (4.75,6.25) {\tiny{$b_1$}};
\draw (6.25,6.25) -- (5.75,5.75)  -- (5.25,6.25) -- (5.75,6.75) --  (6.25,6.25); \node at (5.75,6.25) {\tiny{$b_1$}};
\node at (7, 6.25) {\tiny{row $k-1$}};

\node at (-5.25, 5.75) {{$\ddots$}};
\node at (5.25, 5.75) {{$\iddots$}};
 
\draw (-4.5,4.5) -- (-4,5) -- (-4.5,5.5) -- (-5,5) --  (-4.5,4.5); \node at (-4.5,5) {\tiny{$b_{\ell-p}$}};
\node at (-3.25, 5) {{$\cdots$}};
\node at (1.5, 5) {{$\cdots$}};
\draw (3,5) -- (2.5,4.5)  -- (2,5) -- (2.5,5.5) --  (3,5); \node at (2.5,5) {\tiny{$b_{\ell-p}$}};
\draw (4,5) -- (3.5,4.5)  -- (3,5) -- (3.5,5.5) --  (4,5); \node at (3.5,5) {\tiny{$b_{\ell-p}$}};
\draw  (4.5,4.5) -- (4,5) -- (4.5,5.5) -- (5,5) --  (4.5,4.5); \node at (4.5,5) {\tiny{$b_{\ell-p}$}};
\node at (5.9,5) {\tiny{row $k-\ell+p$}};
 
\draw (-4,4) -- (-3.5,4.5) -- (-4,5) -- (-4.5,4.5) --  (-4,4); \node at (-4,4.5) {\tiny{$b_{\ell-p+1}$}};
\node at (-2.75, 4.5) {{$\cdots$}};
\node at (1, 4.5) {{$\cdots$}};
\draw (2.5,4.5) -- (2,4)  -- (1.5,4.5) -- (2,5) --  (2.5,4.5); \node at (2,4.5) {\tiny{$b_{\ell-p+1}$}};
\draw (3.5,4.5) -- (3,4)  -- (2.5,4.5) -- (3,5) --  (3.5,4.5); \node at (3,4.5) {\tiny{$b_{\ell-p+1}$}};
\draw[blue] (4,4) -- (3.5,4.5) -- (4,5) -- (4.5,4.5) --  (4,4); \node at (4,4.5) {\tiny{$a_{\ell-p+2}$}};
 \node at (5.6,4.5) {\tiny{row $k - \ell+p-1$}};

\node at (-3.5, 4) {{$\ddots$}};
\node at (3.5, 4) {{$\iddots$}};

\draw (-2.75,2.75) -- (-2.25,3.25) -- (-2.75,3.75) -- (-3.25,3.25) --  (-2.75,2.75); \node at (-2.75, 3.25) {\tiny{$b_{\ell-2}$}};
\node at (-1.5, 3.25) {{$\cdots$}};
\node at (-0.25, 3.25) {{$\cdots$}};
\draw (1.25,3.25) -- (.75,2.75)  -- (.25,3.25) -- (.75,3.75) --  (1.25,3.25); \node at (.75,3.25) {\tiny{$b_{\ell-2}$}};
\draw (2.25,3.25) -- (1.75,2.75)  -- (1.25,3.25) -- (1.75,3.75) --  (2.25,3.25); \node at (1.75,3.25) {\tiny{$b_{\ell-2}$}};
\draw[blue] (2.75,2.75) -- (2.25,3.25) -- (2.75,3.75) -- (3.25,3.25) --  (2.75,2.75); \node at (2.75,3.25) {\tiny{$a_{\ell-1}$}};
\node at (4.1,3.25) {\tiny{row $k-\ell +2$}};

\draw (-2.25,2.25) -- (-1.75,2.75) -- (-2.25,3.25) -- (-2.75,2.75) --  (-2.25,2.25); \node at (-2.25, 2.75) {\tiny{$b_{\ell-1}$}};
\node at (-1, 2.75) {{$\cdots$}};
\draw (.75,2.75) -- (.25,2.25)  -- (-.25,2.75) -- (.25,3.25) --  (.75,2.75); \node at (.25,2.75) {\tiny{$b_{\ell-1}$}};
\draw (1.75,2.75) -- (1.25,2.25)  -- (.75,2.75) -- (1.25,3.25) --  (1.75,2.75); \node at (1.25,2.75) {\tiny{$b_{\ell-1}$}};
\draw[blue] (2.25,2.25) -- (1.75,2.75) -- (2.25,3.25) -- (2.75,2.75) --  (2.25,2.25); \node at (2.25,2.75) {\tiny{$a_\ell$}};
\node at (3.65,2.75) {\tiny{row $k - \ell+1$}};

\draw[blue] (-1.75,1.75) -- (-1.25,2.25) -- (-1.75,2.75) -- (-2.25,2.25) --  (-1.75,1.75); \node at (-1.75, 2.25) {\tiny{$a_{\ell+1}$}};
\node at (-0.5, 2.25) {{$\cdots$}};
\draw[blue]  (.75,1.75) -- (.25,2.25) -- (.75,2.75) -- (1.25,2.25) --  (.75,1.75); \node at (.75,2.25) {\tiny{$a_{\ell+1}$}};
\draw[blue]  (1.75,1.75) -- (1.25,2.25) -- (1.75,2.75) -- (2.25,2.25) --  (1.75,1.75); \node at (1.75,2.25) {\tiny{$a_{\ell+1}$}}; 
\node at (3,2.25) {\tiny{row $ k -\ell$}}; 

\node at (-1.25, 1.75) {{$\ddots$}};
\node at (1.25, 1.75) {{$\iddots$}};

\draw[blue] (-.5,.5) -- (0,1) -- (-.5,1.5) -- (-1,1) -- (-.5,.5); \node at (-.5, 1) {\tiny{$a_{k-1}$}};
\draw[blue]  (.5,.5) -- (0,1) -- (.5,1.5) -- (1,1) -- (.5,.5); \node at (.5,1) {\tiny{$a_{k-1}$}}; 
\node at (1.5,1) {\tiny{row $2$}}; 

\draw[blue] (0,0) -- (-.5,.5) -- (0,1) -- (.5,.5) -- (0,0); \node at (0, .5) {\tiny{$a_k$}}; 
\node at (1, .5) {\tiny{row $1$}}; 

\draw[blue] (-2.25,2.25) --  (-1.75,2.75) -- (-1.25, 2.25);
\draw[blue] (.25,2.25) --  (.75,2.75) -- (1.25, 2.25) -- (1.75, 2.75) -- (2.25, 3.25) -- (2.75, 3.75);
\draw[blue] (3.5, 4.5) -- (4,5) -- (4.5,4.5);
\end{tikzpicture}
\]
\caption{The filling of $\gamma_k$ in the proof of Lemma \ref{lemma:LWdiagonals}}
\label{figure:gammakfilling}
\end{figure}
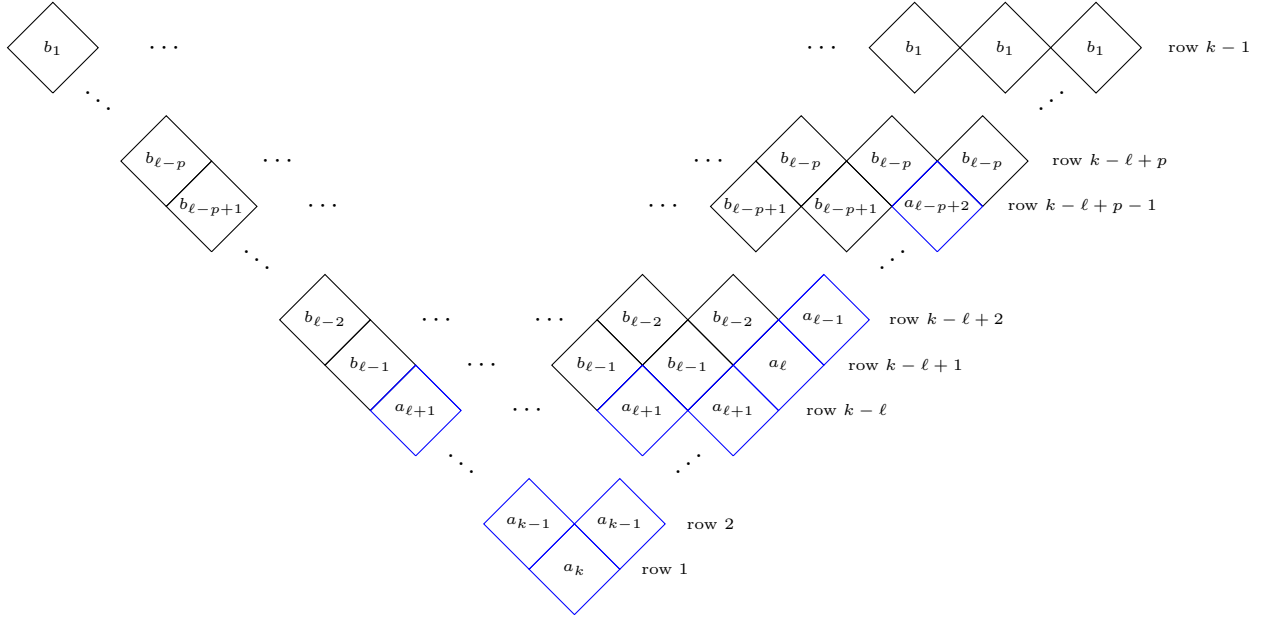

Now, we need to check that the sum of the entries of these $\ell$ fillings is equal to the leftover pieces. For the filling $p$, the sum of the entries is computed below.
\begin{align*}
\sum_{j=1}^{k-1} \sum_{i=1}^j P_{L(i,j)} & = \sum_{j=1}^{k-\ell} j P_{a_{k+1-j}} + \sum_{j=k-\ell+1}^{k-\ell+p-1} \left( (j-1) P_{b_{k-j}} + P_{a_{k-j+1}} \right) + \sum_{j=k-\ell+p}^{k-1} j P_{b_{k-j}} \\
& = \sum_{i=1}^{k-1} (k-i) P_{a_{i+1}} + \sum_{i=1}^{p-1} \left( (k - \ell + i-1) P_{b_{\ell - i}} + P_{a_{\ell-i+1}} \right) + \sum_{i=p}^{\ell-1} (k - \ell +i) P_{b_{\ell -i}} 
\end{align*}
The sum of all the fillings reduces as follows:
\begin{align*}
&\sum_{p=1}^{\ell} \left( \sum_{i=1}^{k-1} (k-i) P_{a_{i+1}} + \sum_{i=1}^{p-1} \left( (k - \ell + i-1) P_{b_{\ell - i}} + P_{a_{\ell-i+1}} \right) + \sum_{i=p}^{\ell-1} (k - \ell +i) P_{b_{\ell -i}} \right) \\
& = \ell \sum_{i=\ell}^{k-1} (k-i) P_{a_{i+1}} + \sum_{i=1}^{\ell-1} \sum_{p=i+1}^{\ell} \left( (k - \ell + i-1) P_{b_{\ell - i}} + P_{a_{\ell-i+1}} \right) + \sum_{i=1}^{\ell-1}\sum_{p=1}^{i} (k - \ell +i) P_{b_{\ell -i}} \\
& = \ell \sum_{i=\ell}^{k-1} (k-i) P_{a_{i+1}} + \sum_{i=1}^{\ell-1} (\ell -i)\left( (k - \ell + i-1) P_{b_{\ell - i}} + P_{a_{\ell-i+1}} \right) + \sum_{i=1}^{\ell-1} i (k - \ell +i) P_{b_{\ell -i}}\\
& = \ell \sum_{i=\ell}^{k-1} (k-i) P_{a_{i+1}} + \sum_{i=1}^{\ell-1} \left( i(k - i-1) P_{b_{i}} + iP_{a_{i+1}} + (\ell-i) (k -i) P_{b_{i}}\right) \\
& = \ell \sum_{i=\ell}^{k-1} (k-i) P_{a_{i+1}} + \sum_{i=1}^{\ell-1} \left( ((k-i)\ell -i) P_{b_{i}} + iP_{a_{i+1}}\right).
\end{align*}
This is exactly the leftover pieces from (\ref{eqn:prooflower}). As the sum of each of these $\ell$ fillings are bounded below by $M_k$, we have shown 
\[
k M_{k-1} + M_k + (k-1) M_{\overline{k+1}} \geq k M_{\overline{k}} + (k-\ell) M_k + \ell M_k.
\]

We now prove that $k M_k + k M_{\overline{k}}$ is equal to the minimum of these two terms. This depends entirely on which of the sequences minimizing $M_k$ and $M_{\overline{k}}$ begins first. Let $(j_1, \dots, j_{k-1})$ and $(\ell_1, \dots, \ell_{k-1})$ be sequences minimizing $M_k$ and $M_{\overline{k}}$ respectively and assume that $j_1 < \ell_1$. 

As $(j_1, \dots, j_{k-1})$ minimizes $\sum_{i=1}^{k-1} (k -i) P_{j_i}$, then it is clear that $(j_2, \dots, j_{k-1})$ minimizes $\sum_{i=2}^{k-1} (k -i) P_{j_i}$ with the condition that $j_1 < j_2$, i.e. $(j_2, \dots, j_{k-1})$ will be a sequence minimizing $M_{\overline{k-1}}(P_{j_1 +1}, \dots, P_m)$. This leads to the equality 
\[
M_k = (k-1) P_{j_1} + M_{\overline{k-1}}(P_{j_1+1}, \dots, P_m).
\]
As $(P_{j_1+1}, \dots, P_m)$ is contained in $(P_{j_1}, \dots, P_m)$, then any valid sequence for $M_{\overline{k-1}}(P_{j_1+1}, \dots, P_m)$ is a valid sequence for $M_{\overline{k-1}}$ as well, so that the second value is smaller than the former.  

As $j_1 < \ell_1$, then  $(j_1, \ell_1, \dots, \ell_{k-1})$ is a valid sequence of $M_{k+1}$ so that $M_{k+1} \leq k P_{j_1} + M_{\overline{k}}$. Then
\begin{align*}
k M_{\overline{k-1}} + M_{\overline{k}} + (k-1) M_{k+1} & \leq k M_{\overline{k-1}} + M_{\overline{k}} + (k-1) (kP_{j_1} + M_{\overline{k}} )\\
& \leq k ((k-1)P_{j_1} + M_{\overline{k-1}}) + kM_{\overline{k}} \\
& \leq k ((k-1)P_{j_1} + M_{\overline{k-1}}(P_{j_1+1}, \dots, P_m)) + kM_{\overline{k}} \\
& = k M_k + k M_{\overline{k}}
\end{align*}
Thus we have equality when $j_1< \ell_1$.  

Note that as $j_1$ is even, $j_1 +1$ is odd so that the parities of the indices in $(P_{j_1+1}, \dots, P_m)$ are the same as the original sequence. In the case of $\ell_1 < j_1$, restricting $M_{k-1}$ to $(P_{\ell_1+1}, \dots, P_m)$ requires us to take $M_{\overline{k-1}}(P_{\ell_1+1}, \dots, P_m)$ so that parities of the indices in the original sequence are correct. The proof then follows identically. 
\end{proof}

\begin{remark}
For finite MV polytopes, \cite{MVpolytopes} notes that the tropical pl\"{u}cker relations can be viewed as the naive tropicalization of the generalized Pl\"{u}cker relations of \cite{Tensorproductmultiplicities} on the generalized minor functions. 

In the affine case, these diagonal relations take the place of the tropical Pl\"{u}cker relations. Our initial attempt to prove these equations was to find positive functions which tropicalize to these diagonals, and thus find an analog of the Pl\"{u}cker relations. 

Unfortunately, the detropicalizated diagonal relations do not obviously tropicalize to the diagonal relations. To demonstrate this, we find the detropicalized diagonal relation for $k=2$ on the generalized minors by studying the representation theory of $\widehat{SL_2}$. 

Define the map $K_{0,1} : V(\omega_0) \otimes V(\omega_1) \rightarrow V(\omega_0 + \omega_1)$. In the subrepresentation $V(\omega_0 + \omega_1 - \delta)$ of the domain, we have the weight vectors
\begin{center}
$\omega_0 \cdot F_0F_1\omega_1 + F_1F_0\omega_0 \cdot \omega_1 - 2 s_1\omega_1 \cdot s_0 \omega_0,$

$\omega_0 \cdot F_1F_0F_1\omega_1 + 2 \omega_1 \cdot s_1s_0\omega_0 - F_1F_0\omega_0 \cdot s_1 \omega_1, $

$F_0F_1F_0\omega_0 \cdot \omega_1  + 2 \omega_0 \cdot s_0s_1\omega_1 - s_0\omega_0 \cdot F_1F_0F_1\omega_1.$
\end{center}
Thus, the representation $V(\omega_0) \otimes V(\omega_1) \otimes \text{ker}(K_{01})$ contains the vector
\begin{equation}\label{diagvector}
\begin{aligned}
v = &2(\omega^2_0 \cdot s_1\omega_1 \cdot s_0s_1\omega_1+\omega^2_1 \cdot s_0\omega_0 \cdot s_1s_0\omega_0-  s_0\omega_0^2\cdot s_1\omega_1^2) \\
&+ \omega_0\cdot\omega_1 (s_0\omega_0\cdot F_1F_0F_1\omega_1+ F_0F_1F_0\omega_0\cdot s_1\omega_1).
\end{aligned}
\end{equation}
As the kernel of the map $K: V(\omega_0)^{\otimes 2}\otimes V(\omega_1)^{\otimes 2} \rightarrow V(2\omega_0 + 2\omega_1)$ contains $V(\omega_0) \otimes V(\omega_1) \otimes \text{ker}(K_{01})$, the vector (\ref{diagvector}) is in ker$(K)$. Note that this vector is not contained in a irreducible subrepresentation of the kernel of $K$.

Let $x \in V(\omega_0)^{\otimes 2}\otimes V(\omega_1)^{\otimes 2}$. Consider the functions $\Delta_x: \widehat{SL_2} \rightarrow \C $ by
\[
\Delta_x(g) = \langle g \cdot x, v_{\omega_0}^2\otimes v_{\omega_1}^2 \rangle.
\]
For $x \in \ker(K)$, $\Delta_x =0$. By using $v$ from (\ref{diagvector}), this reduces to the following equation of generalized minors:
\begin{align}\label{untropdiag}
\Delta^2_1 \Delta_{2} \Delta_{\bar{3}}+  \Delta^2_{\bar{1}} \Delta_{\bar{2}} \Delta_{3} + \frac{1}{2}\Delta_1\Delta_{\bar{1}}(\Delta_{2} \Delta_{F_1F_0F_1\omega_1} + \Delta_{F_0F_1F_0\omega_0}\Delta_{\bar{2}})= \Delta^2_{2} \Delta^2_{\bar{2}} .
\end{align}
Note that the vectors $F_1F_0F_1\omega_1$ and $F_0F_1F_0\omega_0$ have weight which is not an extremal weight. As shown in Lemma \ref{lemma:LWdiagonals}, these terms will not contribute to the tropicalization, though the reason why is not clear.
\end{remark}

\subsection{Non-negative tropical points}

Let $(i_1, \dots, i_m)$ be a reduced word for $w$. Set $w_{k}^\ui = s_{i_1} s_{i_2} \cdots s_{i_k}$ for $1 \leq k \leq m$ and set $w_0^\ui = e$. Define the potential function on $L^{w^{-1}}$ as
\begin{align*}
\tau &= \Delta_{\omega_{i_1}}^{-1} \Delta_{s_{i_1}\omega_{i_1}}^{-1} \Delta_{\omega_{i_2}}^2 + \Delta_{\omega_{i_2}}^{-1} \Delta_{w_2^\ui \omega_{i_2}} ^{-1} \Delta_{s_{i_1}\omega_{i_1}}^{2} + \sum_{k=2}^{m-1} \Delta_{w_{k-1}^\ui\omega_{i_{k-1}}}^{-1} \Delta_{w_{k+1}^\ui \omega_{i_{k+1}}} ^{-1} \Delta_{w_k^\ui\omega_{i_k}}^{2}.
\end{align*}
In the Lusztig coordinates, as the generalized minors are positive functions by Lemma \ref{lemma:positiveminors} and $\tau$ is a subtraction-free expression of these minors, then $\tau$ is also a positive function. We will consider the non-negative tropical points $L^{w^{-1}}(\Z^\trop)_\geq^\tau = \{ \ell \in L^{w^{-1}}(\Z^\trop): \tau^\trop(\ell) \geq 0\}$, or more explicitly, this is the set of $\ell \in L^{w^{-1}}(\Z^\trop)$ such that
\begin{align*}
0 \leq \min_{k=2, \dots,m-1} \{ & 2 \Delta^\trop_{\omega_{i_2}}(\ell)-\Delta^\trop_{\omega_{i_1}}(\ell)-\Delta^\trop_{s_{i_1} \omega_{i_1}}(\ell),  2\Delta^\trop_{s_{i_1}\omega_{i_1}}(\ell)-\Delta^\trop_{\omega_{i_2}}(\ell) -\Delta^\trop_{w_2^\ui\omega_{i_2}}(\ell),\\
& 2\Delta^\trop_{w_k^\ui\omega_{i_k}}(\ell)-\Delta^\trop_{w_{k-1}^\ui\omega_{i_{k-1}}}(\ell)- \Delta^\trop_{w_{k+1}^\ui\omega_{i_{k+1}}} (\ell)\}.
\end{align*}
We want to show that, under the correct coordinates, we can associate $L^{w^{-1}}(\Z^\trop)_{\geq}^\tau \cong \N^m$. First, we show that using the terms of $\tau$, we can define a different set of coordinates on $L^{w^{-1}}$. 

\begin{lemma}\label{lemma:nonnegativetropical}
The map $\phi: L^{w^{-1}} \rightarrow (\C^\times)^m$ defined by 
\[
\phi(g) = \left(\frac{\Delta^2_{\omega_{i_2}}(g)}{\Delta_{\omega_{i_1}}(g) \Delta_{s_{i_1} \omega_{i_1}}(g)}, \frac{\Delta^2_{s_{i_1}\omega_{i_1}}(g)}{\Delta_{\omega_{i_2}}(g)\Delta_{w_2^\ui\omega_{i_2}}(g)}, \frac{\Delta^2_{w_2^\ui\omega_{i_2}}(g)}{\Delta_{s_{i_1} \omega_{i_1}}(g) \Delta_{w_3^\ui \omega_{i_3}}(g)}, \dots, \frac{\Delta^2_{w_{m-1}^\ui \omega_{i_{m-1}}}(g)}{\Delta_{w_{m-2}^\ui \omega_{i_{m-2}}}(g) \Delta_{w_m^\ui \omega_{i_{m}}}(g)} \right)
\]
is a birational isomorphism. 
\end{lemma}

\begin{proof}
By \cite[Theorem 3.3, Proposition 9.1]{Kac-MoodyGroups}, for the reduced word $(i_1, \dots, i_m)$ of $w$, the tuple 
\[
(\Delta_{s_{i_1} \omega_{i_1}}, \Delta_{s_{i_1} s_{i_2} \omega_{i_2}}, \dots,  \Delta_{w_{m-1}^\ui \omega_{i_{m-1}}}, \Delta_{w_m^\ui \omega_{i_m}} )
\]
is an initial seed for the cluster structure of $\C[L^{w^{-1}}]$. This seed induces an isomorphism on the function fields of $L^{w^{-1}}$ and $(\C^\times)^m$. 
Since the map between the function fields is an isomorphism, the map $\psi_1: L^{w^{-1}} \rightarrow (\C^\times)^m$ given by 
\[
\psi_1(g) = \left( \Delta_{w_k^\ui \omega_{i_k}}(g) \right)_{k=1}^m
\]
is also a birational isomorphism.

Consider the map $\psi_2: (\C^\times)^m \rightarrow (\C^\times)^m$ by 
\[
\psi_2(x_1, \dots, x_m) \mapsto \left( \frac{1}{x_1}, \frac{x_1^2}{x_2}, \frac{x_2^2}{x_1x_3}, \dots, \frac{x_k^2}{x_{k-1}x_{k+1}}, \dots, \frac{x^2_{m-1}}{x_{m-2}x_{m+1}} \right)
\]
This map is an isomorphism so the composition $\psi_2 \circ \psi_1$ is still a birational isomorphism. As $\Delta_{\omega_i} =1$ on $L^{w^{-1}}$, then $\psi_2 \circ \psi_1 = \phi$. 
\end{proof}

By this lemma, we have that $\phi^{-1}$ gives a positive structure on $L^{w^{-1}}$. Since $\phi \circ x_{\ui}$ is a positive map, then $x_{\ui}$ and $\phi^{-1}$ are different coordinates in the same positive atlas on $L^{w^{-1}}$. Using the coordinates of $\phi^{-1}$, then $\ell \in L^{w^{-1}}(\Z^\trop)_\geq^\tau$ if and only if $\ell \in \N^m$ by the definition of $\phi$.

\begin{corollary}\label{lemma:tropicallusztig}
$L^{w^{-1}}(\Z^\trop)^\tau_\geq \cong \N^m$.
\end{corollary}

We can show that the collection $(M_\gamma)_{\gamma \in \Gamma^w}$ satisfies the edge inequalities \ref{condition:positivity} on $L^{w^{-1}}(\Z^\trop)_\geq^\tau$.

\begin{lemma}\label{lemma:LWpositivity}
The collection $(M_\gamma)_{\gamma \in \Gamma^w}$ satisfies the edge inequalities on  $L^{w^{-1}}(\Z^\trop)_\geq^{\tau}$. More explicitly, for $\ell \in L^{w^{-1}}(\Z^\trop)_\geq^\tau$, 
\begin{enumerate}[label=(\roman*)]
 \item\label{condition:firstedgeequality} if $\gamma_{k+1} \in \Gamma^w$, then $M_{k-1}(\ell) + M_{k+1}(\ell) \leq 2M_{k}(\ell)$.
 \item\label{condition:secondedgeequality} if $\ogamma_{k+1} \in \Gamma^w$, then $ M_{\overline{k-1}}(\ell) + M_{\overline{k+1}}(\ell) \leq 2M_{\overline{k}}(\ell)$.
\end{enumerate}
\end{lemma}

\begin{proof}
Without loss of generality, assume $w = w_m$. For convenience, we will denote $a_k = 2M_k - M_{k-1} - M_{k+1}$ and $\bar{a}_k = 2M_{\overline{k}} - M_{\overline{k-1}} - M_{\overline{k+1}}$. 

In terms of chamber weights, $\tau = \Delta_{\ogamma_1}^{-1} \Delta_{\gamma_2}^{-1} \Delta_{\gamma_1}^2 + \Delta_{\gamma_1}^{-1} \Delta_{\gamma_3} ^{-1} \Delta_{\gamma_2}^{2} + \sum_{k=2}^{m-1} \Delta_{\gamma_k}^{-1} \Delta_{\gamma_{k+2}} ^{-1} \Delta_{\gamma_{k+1}}^{2}$. 
Thus, $\tau^\trop = \min_{1\leq k \leq m}\{a_k \}$. For $\ell \in L^{w^{-1}}(\Z^\trop)_{\geq}^\tau$ then $\tau^\trop (\ell) \geq 0$ or equivalently, $a_k(\ell) \geq 0$ for $1 \leq k \leq m$. Hence if $\gamma_{k+1} \in \Gamma^w$ the edge inequality \ref{condition:firstedgeequality} holds.

Suppose $\ogamma_{k+1} \in \Gamma^w$. By assumption, the $(k+1)^\text{st}$ diagonal holds:
\begin{align*}
\max \{ & k\bar{a}_{k+1} + (k-1) \bar{a}_{k} + \cdots + 2\bar{a}_3 + \bar{a}_2 - a_1 - 2 a_2 - \cdots - k a_{k},  \\  & k a_{k+1} + (k-1) a_{k} + \cdots + 2 a_3 +  a_2 -\bar{a}_1 - 2\bar{a}_2 - \cdots - k\bar{a}_{k} \} =0.
\end{align*}
If $k=1$, then the second term gives us that $a_2 \leq \bar{a}_1$ and hence $\bar{a}_1$ is positive since $a_2 \geq 0$. Otherwise, by the $k^\text{th}$ diagonal, one of the following equalities holds:
\begin{align}
 (k-1) \bar{a}_{k} + \cdots + 2 \bar{a}_3 +  \bar{a}_2 &= a_1 + 2a_2 + \cdots + (k-1) a_{k-1}, \label{equation:second}\\
(k-1) a_{k} + \cdots + 2 a_3 +  a_2 &= \bar{a}_1 + 2\bar{a}_2 + \cdots + (k-1)\bar{a}_{k-1}. \label{equation:first}
 \end{align}
 If (\ref{equation:first}) holds, then from the second term in the maximum of the $(k+1)^\text{st}$ diagonal, $k a_{k+1}\leq k \bar{a}_{k}$. Thus $\bar{a}_{k}$ is positive as we have already showed that $a_{k+1} \geq 0$.

Suppose that (\ref{equation:second}) holds. This is equivalent to 
\[
 (k-1) \bar{a}_{k} = -( (k-2)\bar{a}_{k-1} + \cdots + 2 \bar{a}_3 +  \bar{a}_2 - a_1 - 2a_2 - \cdots - (k-2) a_{k-2}) +  (k-1) a_{k-1}.
\]
When $k =2$, this equation gives $\bar{a}_2 = a_1$ and hence $\bar{a}_2$ is positive. Otherwise, by the $(k-1)^\text{th}$ diagonal, we have the inequality
\[
(k-2)\bar{a}_{k-1} + \cdots + 2\bar{a}_3 + \bar{a}_2 - a_1 - 2 a_2 - \cdots - (k-2) a_{k-2} \leq 0
\]
and hence $(k-1)\bar{a}_{k} \geq(k-1) a_{k-1}\geq 0$ so $\bar{a}_k$ is positive. Thus if $\ogamma_{k+1} \in \Gamma^w$, then \ref{condition:secondedgeequality} holds. 
\end{proof}

\begin{figure}[ht] 
 \begin{subfigure}{\linewidth}
\[
\begin{tikzpicture}
\node[vertex] (U1) at (2.4, 0.6)[label=below right:$\mu_{k-2}$] {}; 
\node[vertex] (U2) at (3, 1.05)[label=below right:$\mu_{k-1}$] {}; 
\node[vertex] (U3) at (4.2, 2.55)[label=below right:$\mu_k$] {}; 
\node[vertex] (U4) at (4.8, 3.6) [label=below right:$\mu_{k+1}$] {}; 
\node[vertex] (W3) at (-4.2, 2.25) [label=below left:$\omu_k$]{}; 
\node[vertex] (W2) at (-3, 0.75)[label=below left:$\omu_{k-1}$] {}; 

\draw[thick] (U1) -- (U2) -- (U3) -- (U4);
\draw[thick] (W3) -- (W2);

\draw[thick, dotted, blue] (-3.6,1.5) -- (U4);
\draw[thick, blue] (W2) -- (U3);

\end{tikzpicture}
\]
\caption{Using the $k+1$ and $k$ diagonals}
\end{subfigure}
\begin{subfigure}{\linewidth}
\[
\begin{tikzpicture}
\node[vertex] (U1) at (2.4, 0.6)[label=below right:$\mu_{k-2}$] {}; 
\node[vertex] (U2) at (3, 1.05)[label=below right:$\mu_{k-1}$] {}; 
\node[vertex] (U3) at (4.2, 2.55)[label=below right:$\mu_k$] {}; 
\node[vertex] (U4) at (4.8, 3.6) [label=below right:$\mu_{k+1}$] {}; 
\node[vertex] (W3) at (-4.2, 2.85) [label=below left:$\omu_k$]{}; 
\node[vertex] (W2) at (-3, 1.35)[label=below left:$\omu_{k-1}$] {}; 

\draw[thick] (U1) -- (U2) -- (U3) -- (U4);
\draw[thick] (W3) -- (W2);

\draw[thick, dotted, blue] (-3.6, 2.1) -- (U1);
\draw[thick, blue] (W3) -- (U2);

\end{tikzpicture}
\]
\caption{Using the $k$ and $k-1$ diagonals}
\end{subfigure}
\caption{Visualization of the proof of Lemma \ref{lemma:LWpositivity}}
\label{figure:pictureofproof}
\end{figure}
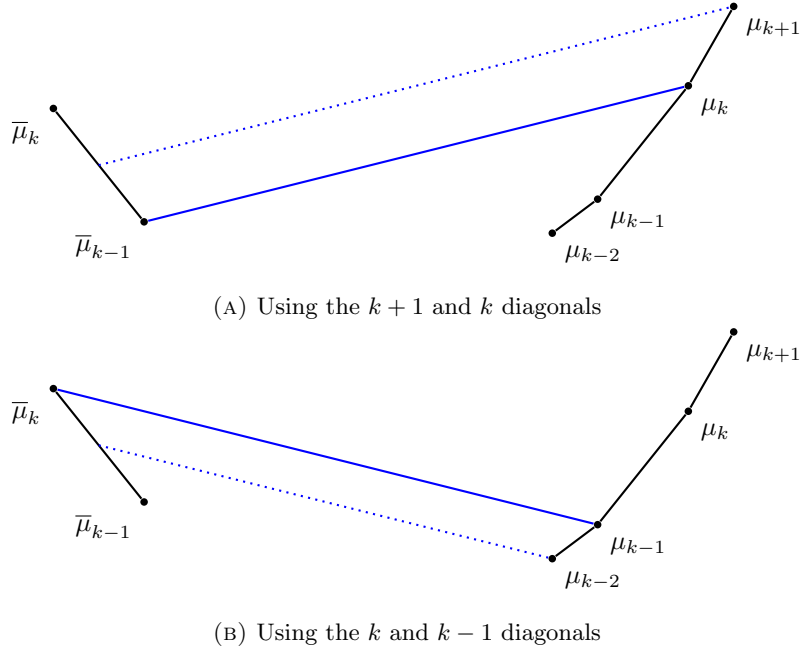

Recall that the diagonals  can be described as follows. Written as in the above proof, the first term in the $k^\text{th}$ diagonal relation guarantees that $k^\text{th}$ vertex on the left side of the polytope is on or below the line in $\alpha_0$ direction from the $k-1^\text{st}$ vertex on the left side. The second term guarantees that the $k-1^\text{st}$ vertex on the left side is on or above the line in the $-\alpha_1$ direction from the $k^\text{th}$ vertex on the right side. The proof uses two diagonals to prove that $\bar{a}_k$ is larger than one of the edges on the left side of the polytope. The two cases can be visualized in Figure \ref{figure:pictureofproof}.  

If $\gamma \not\in \Gamma^w$, then $M_\gamma$ is $\infty$ as it is the minimum of an empty set. We want to redefine the minors $\Delta_\gamma$ so that the tropicalized functions $M_\gamma$ are a BZ datum. 

Inspired by the finite case, for each $u \in W$, we define $u_w$ as the maximal length element such that $u_w \leq_R v$ and $u_w \leq w$. In $\widehat{SL_2}$,  it is easy to see this is well-defined and unique:
\begin{center}
\begin{tabular}{ccc}
 \multicolumn{1}{l}{if $w = w_m$, then} && \multicolumn{1}{l}{if $w = \ow_m$, then} \\
 $u_w = \begin{cases} u, &\text{ if } u \leq w \\
 w_m
 , & \text{ if } u = w_k \text{ and } u \not\leq w\\
 \ow_{m-1}
 , & \text{ if } u = \ow_k \text{ and } u \not\leq w\end{cases}$
 & {\color{white}s} &
 $u_w = \begin{cases} u, &\text{ if } u \leq w \\
 w_{m-1}
 , & \text{ if } u = w_k\text{ and } u \not\leq w\\
 \ow_{m}
 , & \text{ if } u = \ow_k\text{ and } u \not\leq w \end{cases}$
\end{tabular}
\end{center}
This function takes the labelling on vertices above the highest vertex $w$ and maps them to the largest Weyl group element smaller than $w$ on the same side of the polytope. For example, when $w = w_m$, this function takes the labels of the vertices on the right side of the polytope above $\mu_w$ and maps them to $w$. On the left side, this function takes the vertices above $\mu_{\ow_{m-1}}$ and maps the label to $\ow_{m-1}$. 

For $u \omega_i \in \Gamma$, we redefine the minors $\Delta^\text{new}_{u\omega_i} := \Delta_{u_w^{-1}u\omega_i, u\omega_i}$.
 Note that when $u\omega_i \in \Gamma^w$, then $\Delta_{u\omega_i}^\text{new} = \Delta_{u\omega_i}$. For $\gamma \not\in\Gamma^w$, we can explicitly write out the new generalized minor functions.  
\begin{definition}
Fix $w \in W$. If $\ogamma_k \not\in \Gamma^w$, we define the minors
\[
\Delta^{\text{new}}_{\overline{k}}(g) = \begin{dcases} \langle g \cdot v_{\gamma_k}, v_{\ow_{m-1}^{-1} \ogamma_k} \rangle, & \text{ if }w = w_m \\
\langle g \cdot v_{\ogamma_k}, v_{\ow_m^{-1} \ogamma_k} \rangle, & \text{ if } w = \ow_m \end{dcases}
\]
If $ \gamma_k \not\in \Gamma^w$, we define the minors
\[
\Delta^{\text{new}}_{k}(g)= \begin{dcases} \langle g \cdot v_{\gamma_{k}}, v_{w_m^{-1}\gamma_{k}} \rangle, & \text{ if }w = w_m \\ \langle g \cdot v_{\gamma_{k}}, v_{w_{m-1}^{-1} \gamma_{k}} \rangle, & \text{ if }w = \ow_m \end{dcases}
\]
\end{definition}

By definition of $u_w^{-1} u \omega_i$, there is only one $\ui$-path from $u\omega_i$ to $u_w^{-1} u\omega_i$ and thus the new generalized minor $\Delta^\text{new}_{u\omega_i, u_w^{-1} u\omega_i}(x_{\ui}(p_1, \dots, p_m))$ is just a product of the $p_i$'s. 

\begin{lemma}\label{lemma:detropicalLWpositivities}
Fix $w \in W$ and consider $x_{i_m}(p_m) \cdots x_{i_1}(p_1) \in L^{w^{-1}}$. If $\ogamma_k \not\in \Gamma^w$, then
\begin{align*}
\Delta^{\textup{new}}_{\overline{k}}(x_{i_m}(p_m) \cdots x_{i_1}(p_1)) = \begin{dcases} \prod_{s=2}^m \frac{p_s^{k-(s-2)}}{(k-(s-2))!}, & \text{ if }w = w_m \\
\prod_{s=1}^m \frac{p_s^{k-(s-1)}}{(k-(s-1))!}, & \text{ if } w = \ow_m \end{dcases}
\end{align*}
If $\gamma_k \not\in \Gamma^w$, then
\begin{align*}
\Delta^{\textup{new}}_{k}(x_{i_m}(p_m) \cdots x_{i_1}(p_1)) =\begin{dcases} \prod_{s=1}^m \frac{p_s^{k-(s-1)}}{(k-(s-1))!}, & \text{ if }w = w_m \\ \prod_{s=2}^m \frac{p_s^{k-(s-2)}}{(k-(s-2))!} , & \text{ if }w = \ow_m \end{dcases}
\end{align*}
\end{lemma}

\begin{proof}
Suppose $w = w_m$ and let $x_{i_m}(p_m) \cdots x_{i_1}(p_1) \in L^{w^{-1}}$. For $\ogamma_k \not\in \Gamma^w$, then $k \geq m$ and the only possible $\ui$-trail is $\ow_{m-1}^{-1} \ogamma_k \subset \ow_{m-2}^{-1} \ogamma_k \subset \cdots \subset \ow_{1}^{-1}\ogamma_k \subset \ogamma_k \subseteq \ogamma_k$, since $i_1=0$. Using this $\ui$-trail, 
\begin{align*}
\Delta^{\text{new}}_{\overline{k}}(x_{i_m}(p_m) \cdots x_{i_1}(p_1))&= \prod_{s=2}^m \frac{p_s^{k-(s-2)}}{(k-(s-2))!}.
\end{align*}

For $\gamma_{k} \not\in \Gamma^w$, then $k \geq m+1$ and the only possible $\ui$-trail in this case is $w_{m}^{-1} \gamma_k \subset w_{m-1}^{-1} \gamma_k \subset \cdots \subset w_1^{-1} \gamma_k \subset \gamma_k$. Using this $\ui$-trail, 
\begin{align*}
\Delta^{\text{new}}_{k}(x_{i_m}(p_m) \cdots x_{i_1}(p_1)) &= \prod_{s=1}^m \frac{p_s^{k-(s-1)}}{(k-(s-1))!}.
\end{align*} 
By switching the roles of $0$ and $1$, the case of $w = \ow_m$ immediately follows. 
\end{proof}

Denote $M_\gamma^\text{new} := (\Delta_\gamma^\text{new})^\trop$. By tropicalizing these formulae, we have the following explicit equations for the new tropical minors. 
\begin{corollary}
Fix $w \in W$. For an arbitrary  $(P_1, \dots, P_m) \in L^{w^{-1}}(\Z^\trop)$, if $\ogamma_k \not\in \Gamma^w$, then
\begin{align*}
M^{\textup{new}}_{\overline{k}}(P_1, \dots, P_m) &= \begin{dcases} \sum_{s=2}^{m}(k-(s-2)) P_{s}, & \text{if }w=w_m \\ \sum_{s=1}^{m}(k-(s-1)) P_{s}, &\text{if }w=\ow_m 
\end{dcases}
\end{align*}
If $\gamma_k \not\in\Gamma^w$, then
\begin{align*}
M^{\textup{new}}_{k}(P_1, \dots, P_m) &= \begin{dcases} \sum_{s=1}^{m}(k-(s-1)) P_{s}, &\text{if }w=w_m \\ \sum_{s=2}^{m}(k-(s-2)) P_{s}, &\text{if }w=\ow_m
\end{dcases}
\end{align*}

\end{corollary}

These new generalized minors now satisfy the edge equalities.

\begin{lemma}
On $L^{w^{-1}}$, the detropical edge equalities are satisfied:
\begin{enumerate}[label=\arabic*)]
 \item If $\ogamma_{k}\not\in \Gamma^w$, then
 \begin{enumerate}[label=\alph*)]
  \item if $w = w_m$, $(\Delta^\textup{new}_{\overline{k}})^2 = \dfrac{k+1}{k-m+2}\Delta^\textup{new}_{\overline{k-1}}\Delta^\textup{new}_{\overline{k+1}}$
  \item if $w = \ow_m$, $ (\Delta^\textup{new}_{\overline{k}})^2 =  \dfrac{k+1}{k-m+1}\Delta^\textup{new}_{\overline{k-1}}\Delta^\textup{new}_{\overline{k+1}}$
 \end{enumerate}
 \item If $\gamma_k \not\in\Gamma^w$, then
 \begin{enumerate}
  \item if $w = w_m$, $(\Delta^\textup{new}_k)^2 = \dfrac{k+1}{k-m+1}\Delta^\textup{new}_{k-1} \Delta^\textup{new}_{k+1}$
  \item if $w = \ow_m$, $ (\Delta^\textup{new}_k)^2 = \dfrac{k+1}{k-m+1}\Delta^\textup{new}_{k-1} \Delta^\textup{new}_{k+1}$
 \end{enumerate}
\end{enumerate}
\end{lemma}

\begin{proof}
Without loss of generality, assume $w = w_m$. If $ \ogamma_k \not\in \Gamma^w$, then $\ogamma_{k+1} \not\in \Gamma^w$ and so Lemma \ref{lemma:detropicalLWpositivities} gives the values of $\Delta_{\overline{k}}$ and $\Delta_{\overline{k+1}}$. If $\ogamma_{k-1} \not\in \Gamma^w$, Lemma \ref{lemma:detropicalLWpositivities} determines $\Delta_{\overline{k-1}}$. Otherwise, $\ogamma_k = \ogamma_m$. As $\ui$-trails are of length at most $m$, there is exactly one $\ui$-trail from $\ogamma_m$ to $\omega_{i_m}$, and so by Lemma \ref{lemma:positiveminors}, we can show that $\Delta_{\overline{k-1}}$ will be a monomial in the $p_i$'s of the same form as in Lemma \ref{lemma:detropicalLWpositivities}. Hence 
\begin{align*}
\Delta_{\overline{k-1}}^\text{new} (x_{\ui}(p_1, \dots, p_m)) & = \prod_{s=2}^{m} \frac{p_s^{k+1-s}}{(k+1-s)!}, &
\Delta_{\overline{k}}^\text{new} (x_{\ui}(p_1, \dots, p_m)) & = \prod_{s=2}^{m} \frac{p_s^{k+2-s}}{(k+2-s)!}, \\
\Delta_{\overline{k+1}}^\text{new} (x_{\ui}(p_1, \dots, p_m)) & = \prod_{s=2}^{m} \frac{p_s^{k+3-s}}{(k+3-s)!}.
\end{align*}
Thus
\begin{align*}
\Delta_{\overline{k-1}}^\text{new} \Delta_{\overline{k+1}}^\text{new} & = \prod_{s=2}^{m} \frac{p_s^{k+1-s}}{(k+1-s)!} \cdot \prod_{t=2}^{m} \frac{p_t^{k+3-t}}{(k+3-t)!}\\
& = \prod_{s=3}^{m+1} \frac{1}{(k+2-s)!} \cdot \prod_{t=1}^{m-1} \frac{1}{(k+2-t)!} \left( \prod_{u=2}^m p_s^{2k+4-2u} \right)\\
& = \frac{k!}{(k+2 - (m+1))!} \left(\prod_{s=2}^{m} \frac{1}{(k+2-s)!}\right) \frac{(k+2 -m)!}{(k+1)!} \left( \prod_{s=2}^{m} \frac{1}{(k+2-s)!}\right) \prod_{u=2}^m p_s^{2k+4-2u} \\
& = \frac{k+2 -m}{k+1}\left( \prod_{s=2}^{m} \frac{1}{(k+2-s)!}p_s^{k+2-s}\right)^2  = \frac{k+2 -m}{k+1}\left( \Delta_{\overline{k}}^\text{new}\right)^2
\end{align*}
If $\gamma_k \not\in \Gamma^w$, then by a similar argument, the generalized minors are given by
\begin{align*}
\Delta_{k-1}^\text{new} (x_{\ui}(p_1, \dots, p_m)) & = \prod_{s=1}^{m} \frac{p_s^{k-s}}{(k-s)!}, &
\Delta_k^\text{new} (x_{\ui}(p_1, \dots, p_m)) & = \prod_{s=1}^{m} \frac{p_s^{k+1-s}}{(k+1-s)!},\\
\Delta_{k+1}^\text{new} (x_{\ui}(p_1, \dots, p_m)) & = \prod_{s=1}^{m} \frac{p_s^{k+2-s}}{(k+2-s)!}.
\end{align*}
Thus
\begin{align*}
\Delta_{k-1}^\text{new} \Delta_{k+1}^\text{new} & = \prod_{s=1}^{m} \frac{p_s^{k-s}}{(k-s)!} \prod_{t=1}^{m} \frac{p_t^{k+2-t}}{(k+2-t)!}\\
& = \prod_{s=2}^{m+1} \frac{1}{(k+1-s)!} \prod_{t=0}^{m-1} \frac{1}{(k+1-t)!}\left(\prod_{u=1}^m p_u^{2k+2-2u}\right)\\
& = \frac{k!}{(k+1 - (m+1))!} \left(\prod_{s=1}^{m} \frac{1}{(k+1-s)!}\right) \frac{(k+1 - m)!}{(k+1)!}\left(\prod_{t=1}^{m} \frac{1}{(k+1-t)!}\right)\prod_{u=1}^m p_u^{2k+2-2u}\\
& =\frac{k+1 - m}{k+1} \left(\prod_{s=1}^{m} \frac{1}{(k+1-s)!}p_u^{k+1-s} \right)^2 =\frac{k+1-m}{k+1} \left(\Delta_k^\text{new} \right)^2
\end{align*}
\end{proof}

Define the set $(M_\gamma)_{\gamma \in \Gamma}$ by $M_\gamma = \Delta_\gamma^\trop$ if $\gamma \in \Gamma^w$ and $M_\gamma = \left(\Delta_\gamma^\text{new}\right)^\trop$ if $\gamma \in \Gamma \setminus \Gamma^w$. By tropicalizing these equations, we immediately show the edge equalities hold.

\begin{corollary}\label{corollary:LWedgeequalities}
The collection $(M_\gamma)_{\gamma \in \Gamma}$ satisfies the edge equalities \ref{condition:edgeequalities} in Lemma \ref{lemma:BZdata} on $L^{w^{-1}}$. More explicitly, for $\ell \in L^{w^{-1}}(\Z^\trop)$, 
\begin{enumerate}[label=(\roman*)]
 \item if $\ogamma_{k} \not\in \Gamma^w$, then  $2M_{\overline{k}}(\ell) = M_{\overline{k-1}}(\ell) + M_{\overline{k+1}}(\ell)$
 \item if $\gamma_{k} \not\in \Gamma_w$, then $ 2M_{k}(\ell) = M_{k-1}(\ell) + M_{k+1}(\ell)$
\end{enumerate}
\end{corollary}

Finally, we can construct a bijection between the set of lower affine MV polytopes and the non-negative tropical points with respect to $\tau$ of the reduced double Bruhat cell. 

\begin{theorem}\label{theorem:tropicalpointsL^w}
For $w \in W$, the map $L^{w^{-1}}(\Z^\trop)_{\geq} \longrightarrow \Pw$ defined by 
\[
\ell \mapsto (M_\gamma(\ell))_{\gamma \in \Gamma}
\]
is a bijection. 
\end{theorem}

\begin{proof}
Consider the functions $(M_\gamma)_{\gamma \in \Gamma}$. For $\gamma \in \Gamma^w$, these functions satisfy the diagonal relations by Lemma \ref{lemma:LWdiagonals} and the edge inequalities by Corollary \ref{corollary:LWedgeequalities}. For $\gamma \in \Gamma \setminus \Gamma^w$, these functions satisfy the edge equalities by Corollary \ref{corollary:LWedgeequalities}. Hence, for any $\ell \in L^{w^{-1}}(\Z^\trop)$, $(M_\gamma(\ell))_{\gamma \in \Gamma}$ will satisfy the conditions \ref{condition:positivity}  - \ref{condition:diagonal}  in Lemma \ref{lemma:BZdata}. Thus by Theorem \ref{theorem:BZ}, $(M_\gamma(\ell))_{\gamma \in \Gamma}$ is the BZ datum of a lower affine MV polytope of highest vertex $w$ and this map is well-defined. 

To show this is in fact a bijection, consider the injective map which sends the BZ data $(M_\gamma(\ell))_{\gamma \in \Gamma}$ to $\left(2M_{s_{i_1}\cdots s_{i_{k}}\omega_{i_k}}(\ell) - M_{s_{i_1} \cdots s_{i_{k+1}} \omega_{i_{k+1}}}(\ell) - M_{s_{i_1} \cdots s_{i_{k -1}} \omega_{i_{k-1}}}(\ell)\right)_{k=1}^m$. By composing this map with the bijection in Lemma \ref{lemma:nonnegativetropical}, there is a map which sends $(M_\gamma(\ell))_{\gamma \in \Gamma} \mapsto \ell$. Thus the map $\Pw \rightarrow L^{w^{-1}}(\Z^\trop)_\geq^\tau$ is a bijection.
\end{proof}

\begin{remark}
When $G$ is a general affine Kac-Moody group, we expect a similar theorem to hold. 

One crucial component to this theory would be the existence of a positive structure on $L^{w^{-1}}$. In the rank 2 case, the positive structure is trivial as there are no transition maps. For the general case, we would need to know that the transition functions $x_{\ui} \circ x_{\uj}^{-1}$ are subtraction free for two different reduced words $\ui$ and $\uj$ of $w$. 
\end{remark}

\begingroup
\renewcommand{\addcontentsline}[3]{}
\subsection*{Acknowledgments}

Most of this work was completed as part of my PhD thesis. As such, I would like to thank my advisor, Joel Kamnitzer, for his guidance and suggestions. I would like to thank Jiuzu Hong, Florian Herzig, Marco Gualtieri, Lisa Jeffrey and Peter Tingley for their suggestions and corrections. I would also like to thank Dinakar Muthiah for helpful conversations. 

During this work, I was partially supported by an NSERC graduate scholarship. 

\endgroup

\bibliographystyle{alpha}
\bibliography{bibliography.bib}  
 
\end{document}